\documentclass[11pt]{article}

\usepackage[utf8]{inputenc}
\usepackage[margin=1in]{geometry}
\usepackage{amsmath,amsthm,amssymb}
\usepackage{mathrsfs}
\usepackage{enumerate}
\usepackage{graphicx}
\usepackage{xfrac}
\usepackage{subfiles}
\usepackage{tikz-cd}
\usepackage{comment}
\usepackage{accents} % use if Aball and Bball are defined below
\usepackage[colorlinks=true,hidelinks]{hyperref}
    \AtBeginDocument{}
\usepackage{tikz} % for tikz pictures
\usetikzlibrary{arrows}
\usetikzlibrary{patterns}
\usepackage{pgfplots} % for graphing functions in tikz
\usetikzlibrary{decorations.markings} % use for arrows in tikz
\usetikzlibrary{3d,calc} %3d plotting in tikz
\usepackage{pifont}
\DeclareFontFamily{OT1}{pzc}{}
  \DeclareFontShape{OT1}{pzc}{m}{it}{<-> s * [1.200] pzcmi7t}{}
  \DeclareMathAlphabet{\mathpzc}{OT1}{pzc}{m}{it}
\usepackage{fancyhdr}
\newcommand{\Z}{\mathbb{Z}}
\newcommand{\Q}{\mathbb{Q}}
\newcommand{\R}{\mathbb{R}}

\newcommand{\F}{\mathcal{F}}
\renewcommand{\L}{\mathcal{L}}
\newcommand{\E}{\mathcal{E}}

\newcommand{\M}{\mathcal{M}}

\newcommand{\D}{\mathcal{D}}
\newcommand{\G}{\mathbb{G}}
\newcommand{\X}{\mathcal{X}}
\newcommand{\Y}{\mathcal{Y}}
\newcommand{\K}{\mathcal{K}}

\newcommand{\Gg}{\mathcal{G}}
\renewcommand{\O}{\mathcal{O}}
\renewcommand{\M}{\mathcal{M}}

\newcommand{\rig}{\mathbf{r}}

\newcommand{\Cc}{\mathscr{C}}
\newcommand{\dotr}[1]{#1^{\bullet}}
\DeclareFontFamily{U}{wncy}{}
    \DeclareFontShape{U}{wncy}{m}{n}{<->wncyr10}{}
    \DeclareSymbolFont{mcy}{U}{wncy}{m}{n}
    \DeclareMathSymbol{\Sha}{\mathord}{mcy}{"58} 

\newcommand{\Spec}{\ensuremath{\operatorname{Spec}}}

\newcommand{\Aut}{\ensuremath{\operatorname{Aut}}}

\newcommand{\Coh}{\ensuremath{\operatorname{Coh}}}

\newcommand{\Hom}{\ensuremath{\operatorname{Hom}}}
\newcommand{\Ext}{\ensuremath{\operatorname{Ext}}}

\DeclareMathOperator{\Mod}{Mod}

\DeclareMathOperator{\Br}{Br}
\DeclareMathOperator{\CH}{CH}

\newcommand{\Pic}{\mathbf{Pic}}

\newcommand{\et}{\ensuremath{\operatorname{\acute{e}t}}}

\DeclareMathOperator{\QCoh}{QCoh}
\newcommand{\gm}{\mathbb{G}_m}
\newcommand{\bgm}{B\mathbb{G}_m}

\renewcommand{\epsilon}{\varepsilon}

\newtheorem{thm}{Theorem}[section]
\newtheorem{prop}[thm]{Proposition}

\newtheorem{lem}[thm]{Lemma}
\newtheorem{defn}{Definition}
\newtheorem{cor}[thm]{Corollary}

\newtheorem{rem}{Remark}
  \let\oldrem\rem
  \renewcommand{\rem}{\oldrem\normalfont}

\newtheorem{ex}[thm]{Example}
  \let\oldex\ex
  \renewcommand{\ex}{\oldex\normalfont}
\theoremstyle{proof}
\theoremstyle{definition}
\newtheorem{notation}[thm]{Notation}

\title{Derived equivalences of gerbey curves}
\author{Soumya Sankar and Libby Taylor}
\date{ }

\begin{document}

\maketitle

\begin{abstract}
    We study derived equivalences of certain stacks over genus $1$ curves, which arise as connected components of the Picard stack of a genus $1$ curve.  To this end, we develop a theory of integral transforms for these algebraic stacks.  We use this theory to answer the question of when two gerbey genus $1$ curves are derived equivalent. We use integral transforms and intersection theory on stacks to answer the following questions: if $C'=Pic^d(C)$, is $C=Pic^f(C')$ for some integer $f$?  If $C'=Pic^d(C)$ and $C''=Pic^f(C')$, then is $C''=Pic^g(C)$ for some integer $g$? 
\end{abstract}

\section{Introduction}

The derived category of a variety is a well-studied invariant. Two projective varieties $X$ and $Y$ over a field $k$ are said to be derived equivalent if there is an equivalence $D(X) \cong D(Y)$ of $k$-linear triangulated categories.  Finding non-isomorphic derived equivalent varieties has become an important problem in algebraic geometry.  If $\K \in D(X \times Y)$ is a complex on $X \times Y$, then we say that the \emph{integral transform with kernel $\K$} is the morphism
\[
\Phi^\K:D(X) \to D(Y)
\]
defined by
\[
\E \mapsto R\pi_{Y*}(L\pi_X^*(\E) \otimes \K).
\]
Orlov's representability theorem in \cite{orlov} proves that any fully faithful exact functor $\Psi: D(X) \to D(Y)$ is an integral functor, and that the kernel is uniquely determined as an element of the derived category by $\Psi$.
For an overview, see \cite{greenbook}.

Bondal and Orlov show in \cite{bondalorlov} that for a variety $X$, if $\omega_X$ is ample or anti-ample, then one can recover $X$ from $D(X)$.  This implies that when $X$ is either Fano or of general type, $X$ has no non-isomorphic derived equivalent varieties. Furthermore, any two derived equivalent varieties have the same dimension.  Therefore, in the case of curves, the only nontrivial case to study is that of derived equivalences of genus $1$ curves. Antieau, Krashen and Ward study in \cite{AKW} the question of when two genus $1$ curves are derived equivalent. In particular, they prove the following theorem.

\begin{thm}[\cite{AKW}, Lemma 2.3.]
\label{lem:akw}
Let $C$ and $C'$ be genus $1$ curves over any field $k$ such that $D(C) \cong D(C')$.  Then $C'$ is the moduli space of degree $d$ line bundles on $C$ for some integer $d$, and $C \times C'$ admits a universal sheaf.
\end{thm}

Thus a necessary condition for the derived equivalence of two curves is the existence of a universal sheaf. It turns out that the existence of such a sheaf is also sufficient, and can be phrased via a cohomological criterion.

\begin{thm}[\cite{AKW}, Theorem 2.5]
\label{thm:akw}
Let $E$ be an elliptic curve, and let $C$ and $C'$ be principal homogeneous spaces for it. Then $D(C) \cong D(C')$ if and only if there exists an automorphism $\phi\in \Aut_k(E)$ and an integer $a$ coprime to the order of $[C']\in H^1_{\et}(\Spec k,E)$ such that $[C]=\phi_*(a[C'])$.
\end{thm}

The cohomological criterion gives the existence of a universal sheaf on $C \times C'$, which in turn can be used to construct the derived equivalence. In the event that $C \times C'$ does not admit a universal sheaf (see Example \ref{example:nouniversalsheaf}), one has to look for derived equivalences at the level of stacks instead. We extend Theorems \ref{lem:akw} and \ref{thm:akw} to this case.

\begin{thm}\label{thm:explicitdim1}
Let $\Cc$ be a $\gm$-gerbe over a genus $1$ curve defined over a field $k$ of characteristic 0. Let $\Cc'=\Pic^d_1(\Cc)$ be the weight 1 substack of  $\Pic^d(\Cc)$ (see Subsection \ref{subsec:2.3} and Definition \ref{defn:weight1pic}) for some $d>0$ such that $\Cc$ and $\Cc'$ are derived equivalent. Then $\Cc$ is the moduli space of rank $d$ degree $1-d$ vector bundles on $\Cc'$, and the universal sheaf for this moduli problem is the dual of the universal degree $d$ line bundle $\L$.%, whose associated integral functor inverts the derived equivalence $\Phi^\L:D(\Cc) \to D(\Cc')$. 
\end{thm}

%(\textcolor{red}{At some point, we should show that if $D(\Cc) \cong D(\Cc')$, then one is $\Pic^d$ of the other - if that is even true.})
This theorem is a direct generalization of Theorem \ref{lem:akw}. As a converse to this theorem, given two $\G_m$-gerbes, one of which is $\Pic^d$ of the coarse space of the other, we use the universal line bundle at the level of gerbes to construct a derived equivalence.

%In order to have a candidate for the kernel of a derived equivalence, their results require that $C \times C'$ admit a universal sheaf.  In the case that there is no universal sheaf, however, there is no obvious candidate for the kernel of a derived equivalence.  Our contribution is to extend the results in \cite{AKW} to the case that the coarse space does not admit a universal sheaf. (In Section~\ref{sec:lowdim}, we provide examples of curves $C$ and coarse spaces $C'$ satisfying this phenomenon.)  In this case, we must instead use the fine moduli stack of degree $d$ line bundles on a genus $1$ curve.  We prove the following theorem.

\begin{thm}\label{thm:gerbeygenus1}
Let $\Cc$ be a gerbey genus $1$ curve over a field $k$ of characteristic 0, and let $\Cc':=\Pic^d_1(\Cc)$ be the fine moduli space of weight $1$ degree $d$ line bundles on $\Cc$ for $d\neq 0$.  Let $\L$ be the universal weight $1$ degree $d$ line bundle on $\Cc \times \Cc'$.  The integral functor
\[
\Phi^\L:D(\Cc) \to D(\Cc')
\]
is an equivalence. 
\end{thm}

%\begin{thm}\label{thm:picdderivedequiv}
%Let $C$ be a genus 1 curve over $k$, and $\Cc':=\Pic^d(C)$ for $d\neq 0$ be the fine moduli space of degree $d$ line bundles on $C$.  Let $\L$ be the pullback of the universal degree $d$ line bundle on $C \times \Cc'$ to $C \times \bgm \times \Cc'$; then the integral functor $\Phi^\L:D(C\times B\G_m) \to D(\Cc')$ is an equivalence when $d>0$, and $\Phi^{\L^\vee}:D(C \times \bgm) \to D(\Cc')$ is an equivalence when $d<0$.
%\end{thm}

%Next, we use the techniques of this theorem to prove the analogue of Theorem \ref{lem:akw}.

The techniques of Theorems \ref{thm:explicitdim1} and \ref{thm:gerbeygenus1} allow us to go further and prove the following moduli-theoretic result.

\begin{thm}\label{thm:compositiondim1}
If $\Cc$ is a $\gm$-gerbe over a genus $1$ curve, $\Cc':=\Pic^d_1(\Cc)$ and $\Cc'':=\Pic^f_1(\Cc')$, then $\Cc''=\Pic^{df-1}_2(\Cc)$.
\end{thm}
%\textcolor{red}{I will fix the weights here in a bit.}{\color{violet} Done!  The 1 is changed to a 2 in the last subscript.}

The stacks that appear in these theorems are all $\gm$-gerbes over genus $1$ curves.  To prove these theorems, we show that much of the machinery of derived categories and functors between them can be transferred to the level of stacks which are $\gm$ gerbes over proper schemes. We produce a theory of integral transforms and prove a strong simplicity criterion and a version of Verdier duality for these stacks.  The theory of integral transforms of stacks is closely related to that of twisted derived equivalences of varieties, which is studied in detail in ~\cite{andrei}.  This relationship is given by the following theorem.  %We will prove that many derived equivalences of stacks arise from certain moduli problems.  This will generalize some known results on both twisted and untwisted derived equivalences of schemes.  In particular, we prove the following theorem:

\begin{thm}
\label{theorem:introduction:maintheorem1}
    Suppose $X$ and $Y$ are two projective varieties with $\alpha \in H^2_{\et}(X,\gm)$ and $\beta \in H^2_{\et}(Y,\gm)$ corresponding to $\gm$-gerbes $\X$ and $\Y$, respectively, and suppose that $\Phi^\K:D(\X) \cong D(\Y)$ is an equivalence, with $\K$ of weight $1$.  Then there is an equivalence of twisted derived categories $D(X,\alpha) \cong D(Y,\beta)$.  Furthermore, %if the equivalence $D(\X) \cong D(\Y)$ is realized by an integral functor $\Phi^\K:D(\X) \to D(\Y)$, then 
    the equivalence $D(X,\alpha) \cong D(Y,\beta)$ is realized by $\Phi^{\rig(\K)}$ where $\rig:\X \times \Y \to X \times Y$ is the rigidification map.
\end{thm}

This theorem relates both the existence of these derived equivalences and the functors used to produce them.The converse of this theorem is not necessarily true. A twisted derived equivalence $D(X, \alpha) \cong D(Y,\beta)$ does not imply $D(\X) \cong D(\Y)$ on the corresponding gerbes. It does however, imply an equivalence of a special subcategory, which we prove in Section 4, Theorem \ref{thm:twistedversusnormalderivedequivalences}.

The paper is organized as follows.  In Section~\ref{sec:gerbeybasics}, we discuss background on twisted sheaves and sheaves on gerbes, and describe the relationship between the two.  In Section~\ref{sec:gerbeyintegraltransforms}, we prove that integral transforms exist for $\gm$-gerbes over proper schemes and that they satisfy many of the same properties as integral transforms of schemes.  In particular, Verdier duality holds; there is a strong simplicity criterion for a kernel of an integral transform that guarantees that the integral transform is fully faithful; and when the canonical bundles are trivial, any fully faithful integral transform is an equivalence.  In Section~\ref{sec:twisted}, we prove Theorem~\ref{theorem:introduction:maintheorem1} and recover some results of C\u ald\u araru on twisted derived categories of schemes.  In Section~\ref{sec:GRR}, we prove Grothendieck-Riemann-Roch for gerbes over proper schemes.  Grothendieck-Riemann-Roch will be used for intersection theory computations in Section~\ref{sec:lowdim}, where we apply the setup of Section~\ref{sec:gerbeyintegraltransforms} to gerbey curves and prove Theorems~\ref{thm:explicitdim1}, ~\ref{thm:gerbeygenus1}, and \ref{thm:compositiondim1}.

\subsection{Notation}

Throughout this paper, $k$ will denote a field of characteristic $0$, and $B\gm$ will denote the stack ${B\gm}$ over $\Spec k$. The Picard stack of a scheme $X$ will be denoted $\Pic(X)$, while the Picard scheme will be denoted $\text{Pic}(X)$. Unless indicated otherwise, roman letters (e.g. $C$) will be used for schemes, while calligraphic ones (e.g. $\Cc$) will be used for stacks. A \emph{gerbey} curve will refer to a $\G_m$ gerbe over a smooth projective curve. Note here that a $\G_m$ gerbe over an $n$-dimensional scheme is an $n-1$ dimensional stack.

\subsection{Acknowledgements}

The authors would like to thank Katrina Honigs, Sarah Frei, and Danny Krashen for very helpful feedback on an earlier draft of this paper.  They would also like to thank Eric Larson and Isabel Vogt for asking questions during the second author's area exam that led to the results in Section~\ref{sec:lowdim}, and Ben Lim for help with several technical points about stacks.  We would also like to thank  Ravi Vakil, Benjamin Antieau, Jordan Ellenberg and Andrei C\u ald\u araru for helpful insights and comments. 

This material is based upon work supported by the National Science Foundation under Grant No. DMS-1928930 while the first author participated in a program hosted by the Mathematical Sciences Research Institute in Berkeley, California, during the Fall 2020 semester.

\section{Sheaves on gerbes and twisted sheaves on schemes}\label{sec:gerbeybasics}
Our work in this paper is closely related to that of C\u ald\u araru and Lieblich on twisted sheaves and twisted derived categories.  In this section, we discuss the relationship between twisted sheaves on a scheme and sheaves on an appropriate stack. Sections~\ref{subsec:twistedsheaves} and ~\ref{subsec:gerbes} deal with standard background on twisted sheaves and gerbes, respectively.  Sections~\ref{subsec:2.3} and~\ref{subsec:2.4} are less standard, and make explicit an identification between sheaves on a gerbe and twisted sheaves on the coarse space.

For more details on twisted sheaves and for proofs of the statements below, see ~\cite{andrei}.  %this?
Throughout this section $X$ will denote a smooth, projective variety, $\Br(X)=H^{2}_{\et}(X,\gm)$ its \'{e}tale cohomological Brauer group and $D(X)$ its bounded derived category of coherent sheaves. 

\subsection{Brauer classes as twisting data for sheaves}\label{subsec:twistedsheaves}

A class $\alpha \in \Br(X)$ can be interpreted as twisting data for sheaves on $X$, where the twisting data is described using \v{C}ech cocyles. Using the cocycle interpretation of \'{e}tale cohomology, $\alpha$ is given by data of a cover $\{U_i \}$ of $X$, along with sections $\{\alpha_{ijk} \in \O_X^\times(U_{ijk})\}$ satisfying the 2-cocyle condition and defined up to a 1-coboundary. 

\begin{defn}
Let $\alpha \in \Br(X)$, and let $\{\alpha_{ijk}\}$ be a \v{C}ech 2-cocycle representing it. An $\alpha$-twisted sheaf $\F$ on $X$ is the data of a sheaf $\F_i$ on each $U_i$ in the cover, together with maps $\tau_{ij}:U_i \times_X U_j \rightarrow U_i \times_X U_j$ satisfying $\tau_{ij} \circ \tau_{jk} \circ \tau_{ik}^{-1}=\alpha_{ijk}$.  
\end{defn}

Conversely, given a collection of sheaves $\{\F_i \}$, where each $\F_i$ is a sheaf on an open set $U_i$, one can construct a Brauer class by setting $\alpha_{ijk} = \tau_{ij} \circ \tau_{jk}\circ \tau_{ik}^{-1}$. Thus, Brauer classes represent the failure of local sheaves to glue together to give a global sheaf.

Given a variety $X$ and an $\alpha \in \check{C}^2(X, \O_X^{*})$, one can form the abelian category $\Mod(X, \alpha)$: the category of $\alpha$-twisted sheaves on $X$. The objects in this category are $\alpha$-twisted sheaves and the morphisms are morphisms of sheaves $\{\F_i \rightarrow \Gg_i \}$ that preserve the classes $\alpha_{ijk}$. 
Further, if two \v{C}ech cocycles $\{\alpha_{ijk}\}$ and $\{\alpha'_{ijk}\}$ define the same Brauer class%(or if one is defined on a refinement of the cover of the other)
, then $\Mod(X, \alpha) \cong {\Mod}(X, \alpha')$. Thus one can indeed make sense of the category $\Mod(X, \alpha)$ for $\alpha \in \Br(X)$. An object $\F$ in this category is called \emph{coherent} if each $\F_i$ is a coherent sheaf. We will let $\Coh(X,\alpha)$ denote the subcategory of $\Mod(X, \alpha)$ consisting of the coherent objects.  The $\alpha$-twisted derived category of $X$ is defined as
\[
D(X, \alpha) := D^b(\Coh(X, \alpha)).
\]

\subsection{Brauer classes and gerbes}\label{subsec:gerbes}

The second interpretation of Brauer classes that we will use is that of gerbes. For a detailed exposition about gerbes, see~\cite[Chapter 12]{olsson}). We only recall some basic notions and properties associated with them. 
For each $\alpha \in H^2_{\et}(X,\gm)$ there exists an algebraic stack $f: \X \rightarrow X$ such that
\begin{enumerate}
    \item each point $x \rightarrow \X$ has automorphism group $\G_m$, and
    \item the rigidification of $\X$ is $X$.
\end{enumerate}

\'Etale locally on $X$, $\X$ is of the form $X \times \bgm$. One may think of the class $\alpha$ as determining the transition functions for the $\bgm$-bundle structure.  Indeed, by~\cite[06QB]{stacks}, there exists an algebraic space $U$, a group algebraic space $G$ which is flat and locally of finite presentation over $U$, and a surjective, flat, locally finitely presented morphism $U \to X$ such that $\X \times_X U \cong [U/G]$ over $U$.

The derived category of an algebraic stack is defined as follows.  Let $\X$ be an algebraic stack.  Let $\Mod(\X)$  denote the category of sheaves of $\O_\X$-modules on the lisse-\'etale site of $\X$. 
Since $\X$ is assumed to be algebraic, it admits a smooth cover $U \to \X$ by a scheme $U$.  A sheaf on $\X$ is said to be coherent if its pullback to $U$ is coherent. Let $\Coh(\X)$ denote the subcategory of $\Mod(\X)$ of coherent sheaves. For general algebraic stacks, the inclusion of $\Coh(\X)$ into $\Mod(\X)$ need not be exact (see e.g.~\cite[06WU]{stacks}). On the lisse-\'etale site, however, the inclusion $\Coh(\X) \hookrightarrow \Mod(\X)$ is exact (see~\cite[07B8]{stacks}). This means that the derived category $D(\X)$ has a description analogous to that for schemes: it is the bounded derived category of complexes of $\O_\X$-modules with coherent cohomology sheaves. This means that, for any Brauer class $\alpha \in \Br(X)$, we can consider the derived category $D(\X)$ for $\X$ a stack over $X$ associated to $\alpha$.

\begin{rem}
When working on the \'etale site of a gerbe, the derived category of $\X$ is easy to define.  When working on other sites (the fppf site, in particular), or on the \'etale site of a stack which is not smooth over its coarse space, defining the derived category of $\X$ may not be so straightforward, since the inclusion of $\Coh(\X) \hookrightarrow{}\Mod(\X)$ is not necessarily exact.  
Even with this complication, integral transforms can still be defined for any stack which is cohomologically proper over its good moduli space. 
\end{rem}

Thus from $\alpha \in \Br(X)$ one can produce two $k$-linear triangulated categories: $D(X,\alpha)$ and $D(\X)$.  There is a close relation between these two interpretations of Brauer classes: $\alpha$-twisted sheaves on $X$ are ``almost'' the same thing as sheaves on $\X$.  More precisely, Lieblich has proved in~\cite[Prop. 2.2.3.3]{max} that there is an equivalence of fibered categories between twisted sheaves on $X$ and a certain subcategory of sheaves on $\X$.  In what follows, we will reprove some of his results; the ideas are mostly the same as those in ~\cite{max}, although the language used here is somewhat different.

\subsection{Weight of a sheaf on a $\gm$-gerbe}\label{subsec:2.3}

A sheaf on the stack $\bgm$ is the data of a $\gm$-representation. This means that for a coherent sheaf on $\X$, each fiber over a closed point of $X$  comes equipped with a $\gm$ action. Let $V$ be a $\G_m$-representation over $k$ and let $t \in \Z$. For $v \in V$, we say $\gm$ acts on $v$ by weight $t$ if for any $\lambda \in \G_m$, $\lambda \cdot v = \lambda^tv$. If $\G_m$ acts by weight $t$ on every $v \in V$, then we say that $V$ is of pure weight $t$ (as opposed to mixed weight). Any mixed-weight representation can be decomposed into the direct sum of its weight spaces. 
Let $\Coh^{t}(\X)$ denote the subcategory of $\Coh(\X)$ consisting of those sheaves on $\X$ whose stalks are weight $t$ $\gm$-representations. We would like to understand the relationship between $\Coh^t(\X)$ and $\Coh(X,\alpha)$ for each $t$, and understand how the $\Coh^t(\X)$ relate to each other as $t$ varies. As a first step, we discuss the relationship between $\Coh(\bgm)$ and $\Coh(\Spec k)$.  

Recall that $\bgm$ is a moduli space for line bundles, so that for a scheme $T$, $\Hom(T,\bgm)\cong \Pic(T)$. We will think of the $T$ points of $\bgm$ as pairs $(T,L)$ with $L \in \Pic(T)$. % That is, each morphism from $T$ to $\bgm$ is determined by a line bundle on $T$.  
A coherent sheaf $F$ on $\bgm$ is determined by the following data: for every pair $(T,L)$ with $L\in \Pic(T)$, a coherent sheaf $F_L$ on $T$, satisfying the descent data for sheaves on a stack. In this case, $F_L$ is the pullback of $F$ along the morphism $L:T \to \bgm$. On the other hand, a coherent sheaf $V$ on $\Spec k$ can be thought of as the collection of pullbacks $V_T$, for each scheme $T \rightarrow \Spec k$.

Let $f:\bgm \to \Spec k$ be the structure map and let $s:\Spec k \to \bgm$ be a section.  
There is a morphism $s^*:\Coh(\bgm) \to \Coh(\Spec k)$ induced by the pullback along $s$.  The pullback $s^*$ has the effect of sending the collection $F_L$ of sheaves to $F_{\O_T}$, the sheaf associated to the trivial line bundle on $T$. One may profitably think of $s^*$ as a ``forgetful map,'' since it forgets the data of $\Pic(T)$ for each test scheme $T$. Further, since $f \circ s$ is the identity on $\Spec k$, $f^*$ is a right inverse to $s^*$ on $\Coh(\Spec k)$. In general, $s^*$ need not admit a left inverse. But we claim that for any $t \in \Z$, $s^*$ restricted to $\Coh^t(\bgm)$ does have a left inverse. 
\begin{lem}
\label{lemma:weighttsection}
For a fixed $t \in \Z$, define the weight $t$ section map,
\begin{align*}
    f_t^*:\Coh(\Spec k) \rightarrow \Coh(B\G_m)\\
    F/T \mapsto ((T,L) \mapsto F \otimes L^{\otimes t}),
\end{align*}
with a weight 0 action on $F$ and weight $1$ action on $L$. Then $f_t^*$ is the inverse of $s^*$ restricted to $\Coh^t(\bgm)$.  In particular for each $t$, $s^*$ induces an equivalence:
$$
\Coh^t(\bgm) \cong \Coh(\Spec k).
$$
\end{lem}

\begin{proof}
Note first that the image of $f_t^*$ is indeed inside $\Coh^t(\bgm)$. Since the weight of a representation is additive on tensor products, $\gm$ acts on $L^{\otimes t}$ with weight $t$. The fact that $f_t^*$ is a right inverse for $s^*$ follows from the interpretation of $s^*$ as a forgetful map. In order to show that it is also a left inverse, it is enough to show that $f_t^*\Coh(\Spec k) \to \Coh^t(\bgm)$ is surjective. To see this, note that for some sheaf $F$ on $\bgm$ on which $\gm$ acts by weight $t$, $F$ may be written as $(F \otimes L^{\otimes -t})\otimes L^{\otimes t}$, so that $F=f_t^*(F \otimes L^{\otimes -t})$. 
 
\end{proof}

\begin{rem}
The weight $1$ section map $f_1^*: \Coh(X,\alpha) \rightarrow \Coh(\X)$ induces a map on the derived categories $Lf_1^*: D(X,\alpha) \rightarrow D(\X)$ via pullback, and by construction of $f_1^*$ the image sits inside $D^1(\X)$.
Let $F$ be a sheaf on $X$, and identify $F$ with an object in $D(X)$ by placing it in degree $k$ for some integer $k$.  Then $Lf_1^*F$ as an element of $D(\X)$ is the sheaf $f_1^*(F)$ in degree $k$.  This can be computed directly via resolutions. Note that it is important here that $F$ is an honest sheaf rather than any complex of sheaves.
\end{rem}

%\textcolor{red}{I just realized that we don't really use $L$ and $R$ everywhere - only sometimes. Should we switch to a more consistent form?} {\color{violet} yes, we really should.  I've endeavored to get all of them switched to L and R, but it would be great if you can make another pass through to make sure that all the $L$s and $R$s are there.}

\subsection{Preliminary results}\label{subsec:2.4}

By Lemma \ref{lemma:weighttsection} and a base change argument, if $\X$ is the trivial gerbe over a scheme $X$, then we have $\Coh^1(\X) \cong \Coh(X)$, where $\Coh^1(\X)$ denotes the subcategory of $\Coh(\X)$ consisting of sheaves of pure weight $1$.  We now show that this equivalence holds for non-trivial gerbes as well.

\begin{prop}\label{prop:equivofcats}
There is a map of abelian categories $s^*:\Coh^1(\X) \to \Coh(X,\alpha)$ which is defined \'etale locally by pulling back along a section $U \to U \times \bgm$, that is an equivalence.
\end{prop}

\begin{proof}
We first show that there is an equivalence of fibered categories over $X$ between $\Coh^1(\X)$ and $\Coh(X,\alpha)$. Since $X$ is locally trivializable, there exists an \'etale cover $\{U_i\}$ of $X$ such that the structure morphism $\X\to X$ is locally of the form $U_i \times \bgm \to U_i$.  Each of the morphisms $U_i \times \bgm \to U_i$ admits a section $s_i: U_i \rightarrow U_i \times \bgm$.  By the discussion above, $\Coh^1(U_i \times \bgm) \cong \Coh(U_i)$.  Therefore the desired equivalence of fibered categories holds over each \'etale open in a cover of $X$.

It remains to show that the local equivalences of fibered categories glue. We claim that the gluing data is provided by $\alpha$. If $F \in \Coh^1(\X)$, then one can define $s^*(F)$ as the collection of sheaves $\{s_i^*(F|_{U_i \times \bgm}) \}_{i}$. Since $F$ is a sheaf on $\X$, by pulling back transition functions, one can see that the collection $\{s_i^*(F|_{U_i \times \bgm}) \}_{i}$ is an $\alpha$-twisted sheaf on $X$. A similar argument shows that $f_1^*$, which is a priori defined locally on $\Coh(X,\alpha)$, does indeed glue to give a global inverse to $s^*$. This gives the desired equivalence of fibered categories on $X$.  To get the equivalence of abelian categories, one can take global sections of each of the fibered categories.

\end{proof}

\begin{rem}
One may define $\tilde{s}^*$, a ``full" version of $s^*$, as a map
$$
\tilde{s}^*:\Coh(\X) \to \underset{\beta \in \Br(X)}{\bigoplus} \Coh(X,\beta).$$ 
This map respects tensor products and $\Hom$s: if $F \in \Coh(\X)$ has weight $t$ and $H$ has weight $t'$, then $F \otimes H$ has weight $t+t'$ and $\tilde{s}^*(F \otimes H) \in \Coh(X,(t-t')\alpha)$. Similarly, $\Hom(F,H)$ has weight $t'-t$, and $\tilde{s}^*(\Hom(F,H))\in \Coh(X,(t'-t)\alpha)$. 
\end{rem}

For later use, we prove that the map $s^*:\Coh^1(\X) \to \Coh(X,\alpha)$ commutes with flat base change.  

\begin{prop}\label{prop:sisflat}
The map of coherent sheaves $s^*:\Coh^1(\X) \to \Coh(X,\alpha)$ is pullback along a flat map of stacks.  In particular, $s^*$ commutes with $\Hom$ and with base change.
\end{prop}

\begin{proof}
Flatness is preserved by descent (see ~\cite[2.2.11(iv))]{ega4}).  Therefore it suffices to show that $s$ is locally a flat morphism. Therefore by base change, it suffices to prove that a section map $s:\Spec k \to \bgm$ is flat. Composing $\Spec k \overset{s}{\to} \bgm \overset{f}{\to} \Spec k$, where $f$ is the structure map, gives the identity map on $\Spec k$. Now, $f$ is flat, since $\bgm$ is a smooth stack over its base $\Spec k$, and the identity map is flat. %Since flatness satisfies the ``$2$ out of $3$ rule,'' this means that 
Therefore $s$ is flat as well. The fact that a flat morphism of stacks commutes with base change and with $\Hom$ is proved in~\cite{hall}.
\end{proof}

\begin{rem}
Lieblich uses the language of characters in proving these equivalences, rather than weights of representations.  His use of characters allows for $G$-gerbes for groups $G\neq \gm$. In this paper, we only need the results for $\gm$-gerbes.
\end{rem}

\begin{rem}
For general $t$, Lieblich shows that the rigidification map induces an equivalence of categories $\Coh^t(\X) \cong \Coh(X,t\cdot \alpha)$. See for instance,~\cite[Section 2.1.3.]{max}. Suppose $\alpha \in H^2(X,\gm)$ is torsion of order $n$.  Then $\Coh(X,t\cdot \alpha)\cong \Coh(X,(t+n)\cdot \alpha)$ for all $t$.  This means that $\Coh(\X)$ has infinitely many graded components which are, in a sense, cyclic of order $n$.
\end{rem}

\begin{lem}
The derived category of $\Coh^t(\X)$, which we call $D^t(\X)$, exists as a faithful triangulated subcategory of $D(\X)$.  Furthermore, all cohomology sheaves in $D^t(\X)$ have weight $t$.%We need to check how this sits inside $D(\X).$
\end{lem}

\begin{proof}
First, note that since $\Coh^t(\X)\cong \Coh(X,t\cdot \alpha)$, the fact that $\Coh(X,t\cdot \alpha)$ is abelian (see ~\cite{andrei}) implies that $\Coh^t(\X)$ is abelian.  Therefore its derived category exists.  Since taking the derived category of an abelian category is covariant, $\Coh^t(\X) \subset \Coh(\X)$ implies that $D^t(\X)$ is a subcategory of $D(\X)$.  To see that all cohomology sheaves in $D^t(\X)$ have weight $t$, recall that the $\G_m$-action on a quotient representation $A/B$ is inherited from the $\G_m$-action on $A$.  Therefore, the cohomology sheaves of a complex of weight $t$ sheaves themselves have weight $t$.

In order to see that the embedding $\iota: D^t(\X) \hookrightarrow D(\X)$ is faithful, we note that for $\E$ and $\F$ two weight $t$ complexes, $\Hom_{D^t(\X)}(\E,\F) \hookrightarrow \Hom_{D(\X)}(\iota(\E),\iota(\F))$.  
\end{proof}

Since $D(\X)$ may be defined analogously to schemes on the lisse-\'etale site (by taking the homotopy category and inverting quasi-isomorphisms), the equivalence $\Coh^t(\X)\cong \Coh(X,t\cdot \alpha)$ of abelian categories induces an equivalence $D^t(\X) \cong D(X,t\cdot \alpha)$ of triangulated categories. The advantage of working with the larger category $D(\X)$ rather than with $D(X,\alpha)$, is that the former is monoidal while the latter is not. In particular, tensor products in the former are well-defined.

\section{Integral transforms for $\gm$-gerbes}\label{sec:gerbeyintegraltransforms}

\subsection{Existence of integral transforms for gerbes}

To study equivalences of derived categories of $\G_m$-gerbes, we first recall that functors between the derived categories of schemes are produced via integral transforms:

\begin{defn}
\label{defn:integraltransformsforschemes}
Let $X$ and $Y$ be schemes over a field, and let $\K \in D(X \times Y)$.  The integral transform with kernel $\K$ is the morphism of categories
\[
\Phi^\K:D(X) \to \D(Y)
\]
\[
\F \mapsto R\pi_{Y*}(\K \otimes L\pi_X^*(\E)).
\]
\end{defn}

We recall Orlov's representability theorem from \cite{orlov}:

\begin{thm}\label{thm:orlovrepresentability}
For $X$ and $Y$ projective varieties over a field, any fully faithful exact functor $\Psi: D(X) \to D(Y)$ is an integral transform with some kernel.  Moreover, the kernel is determined uniquely by the functor $\Psi$.  
\end{thm}

We will set up a theory of integral transforms for derived categories of $\gm$-gerbes over proper schemes and prove that many facts that hold for integral transforms of schemes also hold for integral transforms of such stacks.\footnote{We make no attempt here to prove a version of Orlov's representability theorem for integral transforms of gerbes, so a priori it is possible that there exist fully faithful exact functors between derived categories of gerbes which are not integral transforms.}  

First, we prove that these integral transforms exist; this is not completely trivial!  Let $\X$ and $\Y$ be $\G_m$-gerbes over proper schemes $X$ and $Y$, respectively.  We would like to consider integral functors $D(\X) \to D(\Y)$, analogously to Definition \ref{defn:integraltransformsforschemes}.  Since $\gm$-gerbes are not proper, it is not immediate that the pushforward along the projection map $\X \times \Y \to \Y$ sends coherent sheaves to coherent sheaves.  We fix this problem using Alper's theory of good moduli spaces (\cite{goodmodulispace}).

\begin{defn}
A morphism of stacks $f: \Y_1 \rightarrow \Y_2$ is called \emph{cohomologically proper} if for any coherent sheaf $\F$ on $\Y_1$ and for all $i\ge 0$, $R^if_{*}\F$ is a coherent sheaf on $\Y_2$.
\end{defn}

\begin{prop}\label{prop:cohomologicallyproper}
The projection map $\X \times \Y \to \Y$ is cohomologically proper.  
\end{prop}

\begin{proof}
We recall the definition of a good moduli space, due to Alper (\cite{goodmodulispace}). Let $\M$ be any algebraic stack and $M$ an algebraic space. A morphism $f:\M \to M$ is a good moduli space map if 
\begin{itemize}
    \item $f_*:\QCoh(\M) \to \QCoh(M)$ is exact, and 
    \item$f_*\O_\M \cong \O_M$.
\end{itemize}
Further, a good moduli space morphism respects coherence when source and target are both locally noetherian, i.e. a good moduli space map is cohomologically proper (~\cite[Thm. 4.17(x)]{goodmodulispace}).

We first claim that if $f:\Y \rightarrow Y$ is a $\G_m$-gerbe, then $f$ is a good moduli space map.  Coherence is a local property, and the property of being a good moduli space is preserved by fpqc base change along an algebraic space (see \cite[Prop. 4.7]{goodmodulispace}).  Therefore we may reduce to showing that the structure map $f:\bgm \to \Spec k$ is a good moduli space map.  

The pushforward $f_*:\Coh(\bgm) \to \Coh(\Spec k)$ is exact because $\gm$ has no group cohomology in dimension larger than $0$. In order to prove that $f_*\O_{\bgm}\cong \O_{\Spec k}$, note first that the structure sheaf of $\bgm$ corresponds to the trivial $\gm$-representation. The pushforward corresponds to taking invariants of the trivial representation, which are $1$-dimensional.  The only $1$-dimensional vector bundle on $\Spec k$ is the structure sheaf, which proves that $f:\bgm \to \Spec k$ is a good moduli space morphism.

Consider the fiber diagram

\medskip
\begin{center}
\begin{tikzcd}
\X \times \Y \arrow[r, "f"] \arrow[d] & X \times \Y \arrow[r, "proper"] \arrow[d, "fpqc"] & \Y \arrow[d, "fpqc"] \\
\X \arrow[r, "gms"]                          & X \arrow[r, "proper"]                                   & \Spec k.                  
\end{tikzcd}
\end{center}
\medskip

The composition of two cohomologically proper morphisms is again cohomologically proper.  Therefore, for $\X \times \Y \rightarrow \Y$ to be cohomologically proper, it is enough to show that the map $f$ in the above diagram is.  We wish to again use the fact that the base change of a good moduli space map along an fpqc morphism from an algebraic space is again a good moduli space map. However, we cannot apply this directly to the diagram above, since the vertical map $X \times \Y \to X$ is not a base change along an algebraic space; instead, it's a base change along an algebraic stack. To solve this issue, choose $U \to \Y$ to be a smooth cover of $\Y$ by a scheme $U$. 
%A good moduli space morphism respects coherence when source and target are both locally noetherian, i.e. coherent sheaves push forward to coherent sheaves (~\cite[Thm. 4.17(x)]{goodmodulispace}), so $X \times B\G_m \to X$ is cohomologically proper.  
We can then extend the commutative diagram to the following:

\medskip
\begin{center}
\begin{tikzcd}
\X \times U \arrow[d] \arrow[r]             & X \times U \arrow[d] \arrow[dd, bend right]      &                            \\
\X \times \Y \arrow[r, "f"] \arrow[d] & X \times \Y \arrow[r, "proper"] \arrow[d, "fpqc"] & \Y \arrow[d, "fpqc"] \\
\X \arrow[r, "gms"]                          & X \arrow[r, "proper"]                                   & \Spec k                  
\end{tikzcd}
\end{center}
\medskip

The composition $X \times U \rightarrow X \times \Y \rightarrow X$ is an fpqc morphism from an algebraic space. Thus by \cite[Prop. 4.7]{goodmodulispace}, $\X \times U \to X \times U$ is a good moduli space morphism. Since all objects involved are locally noetherian, $\X \times U \to X \times U$ is cohomologically proper. Coherence can be checked on a cover, in this case the cover $U \to \Y$, so $f:\X \times \Y \to \Y$ is also cohomologically proper.
\end{proof}

This now allows us to define integral functors on $\G_m$-gerbes over proper schemes.

\begin{cor}
Let $\X$ and $\Y$ be $\G_m$-gerbes over proper schemes and let $D(\X)$ and $D(\Y)$ the bounded derived categories of coherent sheaves on them.  For any $\F \in D(\X \times \Y)$, there is a well-defined functor $\Phi^\F:D(\X) \to D(\Y)$ defined by $\E \mapsto R\pi_{\Y*}(L\pi_\X^*(\E) \otimes \F)$.  That is, integral functors exist for these algebraic stacks.
\end{cor}

\begin{proof}
The usual derived tensor of objects in the derived categories of $\X$ and $\Y$ are well-defined. Since both $\X$ and $\Y$ are smooth, the derived pullback of a bounded complex is bounded, so the ordinary pullback gives a functor on bounded derived categories.  By Proposition \ref{prop:cohomologicallyproper}, the projection map $\pi_\Y$ sends coherent sheaves on $\X \times \Y$ to coherent sheaves on $\Y$, which proves the statement.
\end{proof}

\begin{comment}
We use the projection formula to prove that the composition of integral functors has kernel equal to the convolution of the component kernels, as in the case of coherent sheaves on proper schemes:

\begin{prop}\label{prop:convolution}
Let $\K$, $\L$ be coherent sheaves on $X \times Y$ and $Y \times Z$, respectively.  The $\Phi^\L \circ \Phi^\K=\Phi^{\K *\L}$.
\end{prop}

\begin{proof}
Given $\E\in D(X \times Y)$, one has:
\[
\Phi^\L(\Phi^\K(\E))=\pi_{Z*}(\pi_Y^*[\pi_{Y*}(\pi_X^*\E \otimes \K)] \otimes \L)
\]
\[
\cong \pi_{Z*}(\pi_{Y,Z*}(\pi_{XY}^*(\pi_X^*\E \otimes \K)) \otimes \L)
\]
\[
\cong \pi_{Z*}(\pi_{Y,Z*}(\pi_{XY}^*(\pi_X^*\E \otimes \K) \otimes \pi_{YZ}^*\L))
\]
\[
\cong \pi_{Z*}\pi_{XZ,*}(\pi_{XZ}^*\pi_X^*\E \otimes \pi_{XY}^*\K \otimes \pi_{YZ}^*\L
\]
\[
\cong \pi_{Z,*}(\pi_X^*\E \otimes (\L \otimes \K))
\]
\[
= \Phi^{\K * \L}(\E).
\]
The isomorphism in the second line is cohomology and flat base change, which holds for these algebraic stacks by work of Hall in ~\cite{hall}.  The isomorphisms in the 3rd and 5th lines are the projection formula for stacks, proved in ~\cite{hallrydh}.  The 4th line is an isomorphism by the identity $\pi_X \circ \pi_{XY}=\pi_X \circ \pi_{XZ}$ and $\pi_Z \circ \pi_{XZ}= \pi_Z \circ \pi_{YZ}$.
\end{proof}

\end{comment}

\subsection{Adjoints, equivalences and duality}

We would like to apply the following theorem of Bridgeland in ~\cite{bridgeland}:

\begin{thm}\label{thm:bridgeland}
Let $\mathcal{A}$ and $\mathcal{B}$ be triangulated categories and $F:\mathcal{A} \to \mathcal{B}$ a fully faithful exact functor.  Suppose $\mathcal{B}$ is indecomposable, and not every object of $\mathcal{A}$ is isomorphic to 0.  Then $F$ is an equivalence of categories if and only if $F$ has a left adjoint $G$ and a right adjoint $H$ such that for any object $b$ of $\mathcal{B}$,
\[
Hb \cong 0 \implies Gb \cong 0.
\]
\end{thm}

Bridgeland uses this theorem to prove that any fully faithful integral functor $D(X) \to D(Y)$ is an equivalence when $X$ and $Y$ are Calabi-Yau varieties. The result above is purely in the language of triangulated categories, and we may therefore apply it to prove that any fully faithful integral functor between the derived categories of two gerbes over Calabi-Yau varieties is an equivalence.  To apply  
 Theorem~\ref{thm:bridgeland}, we must show that any integral functor admits left and right adjoints and that those adjoints satisfy the hypotheses of the theorem.  Producing these adjoints will require a gerbey version of Verdier duality to get a left adjoint to $\pi^*$ and a right adjoint to $\pi_*$.

Recall the classical statement of local Verdier duality:

\begin{prop}\label{prop:verdierdualityforschemes}
Let $f:X \to Y$ be a proper map of schemes.  For any sheaves $F$ and $G$ on $X$, there is an isomorphism 
\[
R\Hom(Rf_*F,G) \cong Rf_*R\Hom(F,f^*G\otimes \omega_Y[\dim Y])
\]
of objects in the derived category of $Y$.
\end{prop}

The object $\omega_Y[\dim Y]$ in $D(Y)$ is called the dualizing object on $Y$.  It satisfies the property that $f^*(-)\otimes \omega_Y[\dim Y]$ is a left adjoint to $f^*$.  To state and prove Verdier duality for gerbes, we must first produce a candidate for the dualizing object.

\begin{defn}
Let $\Y \to Y$ a $\gm$-gerbe with $Y$ a proper scheme.  Then $\dotr{\omega_\Y}:=f_0^*(\omega_Y[\dim Y])$ will be called the dualizing object of $\Y$.  That is, $\dotr{\omega_\Y}$ is the pullback of the dualizing object on $Y$ under the weight $0$ section map.  $\dotr{\omega_\Y}$ is an object in $D(\Y)$ equipped with the trivial (i.e., weight zero) $\gm$ action.
\end{defn}

There is an equivalence between $D^0(\Y)$ and $D(Y)$ by Proposition~\ref{prop:equivofcats}.  One may think of $\dotr{\omega_\Y}$ as being the image of $\omega_Y[\dim Y]$ under this equivalence.  By a previous remark, $\dotr{\omega_\Y}$ is concentrated in a single degree, since it is the image under $f_0^*$ of a sheaf concentrated in a single degree in $D(Y)$.  

\begin{prop}\label{lem:verdierduality}
Let $\X$ and $\Y$ be $\G_m$-gerbes over proper schemes $X$ and $Y$ and let $\pi_\X:\X \times \Y \to \X$ the projection map.  Let $\F$ be a coherent sheaf on $\X\times \Y$ and $\mathcal{G}$ a coherent sheaf on $\X$.  Then there is an isomorphism 
\[
R\pi_{\X*}(R\Hom_{\X\times \Y}(\F,\pi_\X^*\mathcal{G}\otimes \dotr{\omega_\Y})) \cong R\Hom_\X(\pi_{\X*}\F,\mathcal{G})
\]
as elements of the derived category of $\X$.
\end{prop}

\begin{proof}

The idea is to reduce to the case of local Verdier duality on $X$ using Proposition \ref{prop:equivofcats}. 

First we must show that the left hand side and right hand side have the same $\gm$-action, since to reduce to the case of local Verdier duality on $X$, we need that the images of both sides under $s^*$ are twisted by the same Brauer class.  

We first reduce to the case that the sheaves $\F$ and $\mathcal{G}$ are of pure weight.  If they are not, then as $\gm$-representations (or equivalently, locally as sheaves) they may be decomposed into the direct sum of their weight spaces.  The statement we are proving is local, and the tensor product and $\Hom$ respect the $\gm$ action.  If each of the direct summands on the left hand and right hand sides are isomorphic, then the whole sheaves are isomorphic.  Therefore we may consider each weight space separately in proving the statement.

Weight of representations is additive on tensor products, and $\dotr{\omega}_\Y$ is defined to be of weight zero.  Therefore the weight of the $\gm$ action on $\pi_\X^*\mathcal{G} \otimes \dotr{\omega}_\Y$ is the same as the weight of the $\gm$ action on $\pi_\X^*\mathcal{G}$.  

Consider the weight of the $\gm$ action on $R\Hom_{\X\times \Y}(\F,\pi_\X^*\mathcal{G}\otimes \dotr{\omega_\Y})\mid_{\{x\}\times \Y}$, since this will give the weight of the action on $R\pi_{\X*}(R\Hom_{\X\times \Y}(\F,\pi_\X^*\mathcal{G}\otimes \dotr{\omega_\Y}))$ restricted to a point $x\in \X$.  Now if $\mathcal{G}$ has weight $t'$ on $\X$, then $\pi_\X^*\mathcal{G}|_{\{x\}\times \Y}$ has weight $t'$ for any point $x\in \X$.  Therefore if $\F\mid_{\{x\}\times \Y}$ has weight $t$, then $R\pi_{\X*}(R\Hom_{\X\times \Y}(\F,\pi_\X^*\mathcal{G}\otimes \dotr{\omega_\Y})$ has weight $t'-t$.  

Similarly, we can consider the weight of the right hand side $R\Hom_\X(\pi_{\X*}\F,\mathcal{G})$.  After restricting to any point $x\in \X$, we see that the weight of $\pi_{\X*}\F|_{\{x\}}$ is the same as the weight of $\F|_{\{x\}}\times \Y$.  If this weight is $t$ and that of $\mathcal{G}$ is $t'$, as above, then the weight of $R\Hom_\X(\pi_{\X*}\F,\mathcal{G})$ is $t'-t$, as before.

Since the weights of the $\gm$ actions on the left and right hand sides are both equal to $t'-t$, the pullback $s^*$ gives sheaves on $X$ twisted by the same Brauer class, $(t'-t)\cdot \alpha$.  That is, there is a map $s^*:D^{t'-t}(\X) \to D(X,(t'-t)\cdot \alpha)$ that sends the left and right hand sides of the equation to elements of $D(X,(t'-t)\cdot \alpha)$.

Let $F:=s^*(\F)$ and $G:=s^*(\mathcal{G})$ be the images of $\F$ and $\mathcal{G}$ on $X$.  The map $s^*$ is a flat base change by~\ref{prop:sisflat}, and $\Hom$ commutes with flat base change on algebraic stacks by~\cite{hall}.  Therefore 
\[
R\Hom_\X(s^*(\pi_{\X*}\F),s^*(\mathcal{G}))\cong s^*(R\Hom_\X(\pi_{\X*}\F,\mathcal{G}))
\]
and
\[
s^*(R\pi_{\X*}(R\Hom_{\X\times \Y}(\F,\pi_\X^*\mathcal{G}\otimes \dotr{\omega_\Y})) \cong R\pi_{X*}(R\Hom_{X\times \Y}(s^*(\F),s^*(\pi_\X^*\mathcal{G}\otimes \dotr{\omega_\Y}))).
\]
By local Verdier duality for twisted sheaves on $X$ (for a proof of this duality, see ~\cite{andrei}), there is an isomorphism of $(t'-t)\cdot \alpha$-twisted sheaves 
\[
R\pi_{X*}(R\Hom_{X \times Y}(F,\pi_X^*G \otimes \omega_Y[\dim Y]) \cong R\Hom_X(\pi_X^*F,G).
\]
An equivalence of categories sends isomorphisms to isomorphisms, so by Proposition~\ref{prop:equivofcats}, the desired isomorphism of elements of the derived category of $\X$ holds.

\end{proof}

\begin{cor}
There is an isomorphism of abelian groups
\[
R\Hom_{\X\times \Y}(\F,\pi_\X^*\mathcal{G}\otimes \dotr{\omega_\Y}) \cong R\Hom_\X(\pi_{\X*}\F,\mathcal{G}).
\]
\end{cor}

\begin{proof}
Take global sections on each side of the isomorphism in Proposition~\ref{lem:verdierduality}.
\end{proof}

Proposition~\ref{lem:verdierduality} gives a right adjoint to the pushforward functor $R\pi_*$, namely $\pi^*(-)\otimes \dotr{\omega}_\Y$. Essentially the same argument will produce a left adjoint to the pullback functor $\pi^*$ which is given by $R\pi_*(-)\otimes \dotr{\omega_\X}$.

We now recall the structure of the argument from ~\cite{bridgeland}.  We want to apply Theorem ~\ref{thm:bridgeland} to conclude that any fully faithful integral functor between $\G_m$-gerbes is an equivalence.  To apply this theorem, we must produce left and right adjoints for $\Phi^\K$ for any kernel $\K$ and prove that they satisfy the hypotheses of Theorem ~\ref{thm:bridgeland}.  The previous lemma produces such adjoints:

\begin{lem}\label{lem:adjoints}
Let $F:D(\X) \to D(\Y)$ be an integral functor with kernel $\K$.  Then $F$ admits a left adjoint 
\[
H=\Phi^{\K^\vee \otimes \pi_\Y^*\omega_\Y[\dim Y]}:D(\Y) \to D(\X)
\]
and a right adjoint 
\[
G=\Phi^{\K^\vee \otimes \pi_\X^*\omega_\X[\dim X]}:D(\Y) \to D(\X).
\]
\end{lem}

\begin{proof}
To produce an adjoint to a composition of functors, it suffices to compose the adjoints of each of the functors involved.  A left adjoint to pullback and right adjoint to pushforward are provided by Lemma~\ref{lem:verdierduality}.  From ~\cite[07BE]{stacks}, we have that $L\pi^*$ and $R\pi_*$ are an adjoint pair in $D(\X)$.  Furthermore, $-\otimes E$ and $-\otimes E^\vee$ are an adjoint pair for $E$ any object in the derived category.  Therefore the statement of the lemma follows by composing the left and right adjoints of all the functors involved in the integral transform and applying Lemma ~\ref{lem:verdierduality}.  
\end{proof}

\begin{defn}
\label{defn:canonicalsheaf}
The canonical sheaf on $\X$ is the pullback of $\omega_X$ along the good moduli space map $\X \to X$.  We say that a $\G_m$-gerbe is Calabi-Yau if its coarse space is Calabi-Yau.
\end{defn}

Define $\omega_\X$ to be the image of $\omega_X$ under the weight $0$ section map.  Since $\dotr{\omega}_X=\omega_X[n]$ and pullback preserves cohomological degree, $\dotr{\omega}_\X\cong \omega_\X[n]$.  Since structure sheaf pulls back to structure sheaf, we note that if $\X$ is Calabi-Yau, then $\omega_\X=\O_\X$, justifying our calling $\X$ a Calabi-Yau stack.  Note that $\dotr{\omega_\X}= \omega_{\X}[\dim X]$ since cohomology commutes with base change along $\X \to X$, and in particular the pullback preserves homological degree.  

Now we conclude that a fully faithful integral functor on these gerbes is an equivalence whenever the coarse spaces have trivial canonical bundle:  

\begin{prop}\label{prop:calabiyauequiv}
Suppose $F:D(\X) \to D(\Y)$ is a fully faithful integral functor.  Then $F$ is an equivalence if and only if $F\O_x \otimes \omega_\Y\cong F\O_x$ for all $x\in \X$, where $\omega_\Y$ is the canonical sheaf of $\Y$ as in Definition \ref{defn:canonicalsheaf}. 
\end{prop}

\begin{proof}
First we argue that it is enough to prove the claim when k is algebraically
closed, again mimicking the argument given in ~\cite{AAFH}. If
\[
L, R: D(\Y) \to D(\X)
\]
are the left and right adjoints of $\Phi^\F$, both induced by $\F^\vee[\dim \X]$, then the
unit of the adjunction $1 \to R\Phi^\F$ and the counit $\Phi^\F L \to 1$ are induced by
maps of kernels, and we want to show that these maps are isomorphisms, or
equivalently that their cones are zero. This is true over $k$ if and only if it is
true over $\overline{k}$.

We follow closely the proof of Theorem 5.4 in ~\cite{bridgeland}.   Let $G$ and $H$ be the left and right adjoints to $F$, respectively.  Suppose first that $F$ is an equivalence.  Then both $G$ and $H$ are quasiinverses, so for each closed point $x\in \X$,
\[
G(F\O_x) \cong H(F\O_x)\cong \O_x.
\]
By Lemma ~\ref{lem:adjoints},
\[
G(F\O_x) \cong G(F\O_x)\otimes \omega_\Y \cong H(F\O_x\otimes \omega_\Y)[\dim X - \dim Y].
\]
Since $G$ is an equivalence, $\dim \X = \dim \Y$ and $F\O_x \otimes \omega_\Y \cong F\O_x$.  Since $Y$ is Calabi-Yau, $\omega_Y=\O_Y$, and $\O_Y$ pulls back to the structure sheaf $\O_\Y$.  Therefore $F\O_x \otimes \omega_\Y=F\O_x$, giving the statement.

Conversely, suppose $\dim \X =\dim \Y = n-1$, so $\dim X = \dim Y =n$, and suppose $F\O_x \otimes \omega_\Y \cong F\O_x$ for all $x\in X$.  We first show that for any objects $A,B\in D(\Y)$, $\Hom^i_{D(\Y)}(A,B)\cong \Hom^{n-i}_{D(\Y)}(A,B)^\vee$.\footnote{Recall that for objects $A$ and $B$ in a triangulated category, $\Hom^i(A,B):=\Hom(A,B[i])$.  When $A$ and $B$ are sheaves on a variety, $\Hom^i(A,B) \cong \Ext^i(A,B)$, see ~\cite[Prop. A.68]{greenbook}.}  Since the pullback $s^*$ along a local section commutes with $\Hom$, we have 
\[s^*\Hom^i_{D(\Y)}(A,B) \cong \Hom^i_{D(Y)}(s^*A,s^*B)
\]
and 
\[s^*\Hom^{n-i}_{D(\Y)}(A,B)^\vee \cong \Hom^{n-i}_{D(Y)}(s^*A,s^*B)^\vee
\]
as elements of $D(Y)$.  By duality on $Y$, 
\[
\Hom^{n-i}_{D(Y)}(s^*A,s^*B)^\vee \cong \Hom^i_{D(Y)}(s^*A,s^*B),
\]
so 
\[
s^*\Hom^{n-i}_{D(\Y)}(A,B)^\vee \cong s^*\Hom^i_{D(\Y)}(A,B).
\]
Now $s^*$ is an equivalence of categories, and therefore sends isomorphisms to isomorphisms, proving that $\Hom^i_{D(\Y)}(A,B)\cong \Hom^{n-i}_{D(\Y)}(A,B)^\vee$.  We now apply this fact:

Let $b\in D(\Y)$ by such that $Hb=0$.  For any closed point $x\in \X$ and $i\in \Z$, we have
\begin{align*}
    \Hom^i_{D(\Y)}(Gb,\O_x) &= \Hom^i_{D(\Y)}(b,F\O_x)\\
    &=\Hom^i_{D(\Y)}(b,F\O_x\otimes \omega_\Y)\\
    &=\Hom^i_{D(\Y)}(b,F\O_x)\\
    &=\Hom^{i}_{D(\Y)}(b,F\O_x)\\
    &=\Hom^{n-i}_{D(\Y)}(F\O_x,b)^\vee\\
    &=\Hom^{n-i}_{D(\X)}(\O_x,Hb)^\vee\\
    &=0.
\end{align*}

Since $\Hom^i_{D(\Y)}(Gb,\O_x)=0$ for all $i\in \Z$ and all closed points $x\in \X$, it must be true that $Gb=0$.  Therefore the functors $G$ and $H$ satisfy the hypotheses of Theorem~\ref{thm:bridgeland}, so the proposition follows.

\end{proof}

For later use, we record a result about how to compose integral transforms of gerbes.

\begin{prop}\label{prop:convolution}
Let $\K$, $\L$ be objects in the derived categories of $\X \times \Y$ and $\Y \times \mathcal{Z}$, respectively.  Let $\pi_{\X\Y}:\X \times \Y \times \mathcal{Z} \to \X \times \Y$ be the projection map, and define $\pi_{\Y\mathcal{Z}}$ and $\pi_{\X\mathcal{Z}}$ similarly.  Define $\K *\L :=R\pi_{\X\mathcal{Z}}(L\pi_{\X\Y}^*\K \otimes L\pi_{\Y\mathcal{Z}}^*\L)$ so that $\K *\L$ is an element of $D(\X \times \mathcal{Z})$.  Then $\Phi^\L \circ \Phi^\K=\Phi^{\K *\L}$.
\end{prop}

\begin{proof}
Given $\E\in D(\X \times \Y)$, one has:
\begin{align*}
    \Phi^\L(\Phi^\K(\E))&=R\pi_{\mathcal{Z}*}(L\pi_\Y^*[R\pi_{\Y*}(L\pi_\X^*\E \otimes \K)] \otimes \L)\\
    & \cong R\pi_{\mathcal{Z}*}(R\pi_{\Y,\mathcal{Z}*}(L\pi_{\X\Y}^*(\pi_\X^*\E \otimes \K)) \otimes \L)\\
    & \cong R\pi_{\mathcal{Z}*}(R\pi_{\Y,\mathcal{Z}*}(L\pi_{\X\Y}^*(L\pi_\X^*\E \otimes \K) \otimes L\pi_{\Y\mathcal{Z}}^*\L))\\
    & \cong R\pi_{\mathcal{Z}*}R\pi_{\X\mathcal{Z},*}(L\pi_{\X\mathcal{Z}}^*L\pi_\X^*\E \otimes L\pi_{\X\Y}^*\K \otimes L\pi_{\Y\mathcal{Z}}^*\L \\
    & \cong R\pi_{\mathcal{Z}*}(L\pi_\X^*\E \otimes (\L * \K))\\
    & =\Phi^{\K * \L}(\E).
\end{align*}

The isomorphism in the second line is cohomology and flat base change, which holds for these algebraic stacks by work of Hall in \cite{hall}.  The isomorphisms in the 3rd and 5th lines are the projection formula for stacks, proved in \cite{hallrydh}.  The 4th line is an isomorphism by the identity $\pi_\X \circ \pi_{\X\Y}=\pi_\X \circ \pi_{\X\mathcal{Z}}$ and $\pi_\mathcal{Z} \circ \pi_{\X\mathcal{Z}}= \pi_\mathcal{Z} \circ \pi_{\Y\mathcal{Z}}$.
\end{proof}

\subsection{Strong simplicity}

Now that we know that fully faithful integral functors between Calabi-Yau stacks are equivalences, we want a criterion in terms of the kernel for determining when an integral functor is fully faithful.  In the case of varieties, this is given by the following theorem of Bondal and Orlov in ~\cite{bondalorlov2}.  Throughout this section, we assume that $k$ is algebraically closed or characteristic zero.  In all applications, we will be able to deduce that the derived equivalence in question holds over $k$ when $k\neq \overline{k}$, as explained in the proof of Proposition~\ref{prop:calabiyauequiv}.

\begin{thm}[\cite{bondalorlov2}, Theorem 1.1]
Let $X$ and $Y$ be smooth projective algebraic varieties of the same dimension, and let $\K$ be a kernel in $D(X \times Y)$.  The functor $\Phi^\K_{X \to Y}$ is fully faithful if and only if $\K$ satisfies
\begin{itemize}
    \item $\Hom^0(\K_x,\K_x)=k$ for all closed points $x\in X$,
    \item $\Hom^i(\K_x,\K_x)=0$ for all $x\in X$ when $i<0$ or $i>\dim X$, and
    \item $\Hom^i(\K_x,\K_y)=0$ for all $i$ and for all closed points $x,y\in X$ with $x\neq y$.
\end{itemize}
If these three criteria are satisfied, then $\K$ is said to be strongly simple over $X$.
\end{thm}

We will prove a similar ``strong simplicity'' criterion for $\gm$-gerbes over such varieties.  To do this, we first recall the definition of closed points on a stack.

\begin{defn}
Two morphisms $p:\Spec K \to \X$ and $q:\Spec L \to \X$, for $K$ and $L$ fields, are said to be equivalent if there exists a field $\Omega$ and a $2$-commutative diagram 

\medskip

\begin{center}
\begin{tikzcd}
\Spec \Omega \arrow[d] \arrow[r] & \Spec K \arrow[d, "p"] \\
\Spec L \arrow[r, "q"]           & \X                    
\end{tikzcd}
\end{center}

\medskip

A closed point on $\X$ is an equivalence class of field-valued points of $\X$ under the above equivalence relation.  
\end{defn}

The set of closed points of a stack will be denoted $|\X|$.  The following description may be used to compute $|\X|$ from a presentation.

\begin{lem}
Let $\X$ be an algebraic stack with presentation $[U/R]$.  The image of $|R|\to |U|\times |U|$ is an equivalence relation, and $|\X|$ is the quotient of $|U|$ by this equivalence relation.
\end{lem}

\begin{proof}
See ~\cite[04XE]{stacks}.
\end{proof}

Note that the set of closed points of $\X$ may be computed \'etale locally on $\X$. Since $\X$ is a gerbe over $X$, it is locally a quotient stack, so such a local presentation exists.
In order to understand the closed points of a $\gm$-gerbe $f: \X \to X$, it suffices to give a presentation for $\X$ and apply the above lemma. Since $\X$ is a $\gm$-gerbe, it is locally of the form  $[X/\gm]$. 
Therefore two morphisms $p:\Spec K \to \X$ and $q:\Spec L \to \X$ are equivalent if $f \circ p(\Spec K) = f\circ q(\Spec L)$, and in this case we write $p \sim q$.  This means that the fiber of a sheaf over a closed point of $\X$ is well-defined.  The closed points on $\X$ are in bijection with closed points on $X$; in particular, a closed point of $\X$ has a well-defined image on $X$ under the rigidification map.  This makes precise the following definition.
 
\begin{defn}
\label{defn:stronglysimple}
Let $\F$ be a sheaf on a product of $\G_m$-gerbes $\X \times \Y$.  $\F$ is said to be strongly simple over $\X$ if
\begin{itemize}
    \item $\Hom^0(\F_\mathbf{x},\F_\mathbf{y})=k$ if $\mathbf{x}$ and $\mathbf{y}$ are closed points of $\X$ with $\mathbf{x}\sim \mathbf{y}$, and 
    \item $\Hom^i(\F_\mathbf{x},\F_{\mathbf{y}})=0$ when $i<0$ or $i>\dim X$ if $\mathbf{x}$ and $\mathbf{y}$ are closed points of $X$ with $\mathbf{x}\sim \mathbf{y}$, and
    \item $\Hom^i(\F_{\mathbf{x}},\F_{\mathbf{y}})=0$ for all closed points $\mathbf{x}$ and $\mathbf{y}$ with $\mathbf{x}\neq \mathbf{y}$ in $\X$.
\end{itemize}
\end{defn}

This turns out to the the right condition to guarantee that the integral functor $\Phi^\F:D(\X) \to D(\Y)$ is fully faithful:

\begin{prop}\label{prop:stackystronglysimple}
Let $\X$ and $\Y$ be $\G_m$-gerbes over proper schemes, and let $\F$ be a sheaf over $\X \times \Y$.  Then $\Phi^\F:D(\X) \to D(\Y)$ is fully faithful if and only if $\F$ is strongly simple over $\X$.
\end{prop}

\begin{proof}

If $\Phi^\F$ is fully faithful and $j$ denotes the inclusion of a closed point into $\X$, then
\begin{equation} 
\begin{split}
\Hom^i_{D(\Y)}(\mathbf{L}j_{\mathbf{x}}^*\F,\mathbf{L}j_{\mathbf{y}}^*\F) & = \Hom^i_{D(\Y)}(\Phi^\F_{\X \to \Y} \O_{\mathbf{x}},\Phi^\F_{\X \to \Y}(\O_{\mathbf{y}})) \\
 & = \Hom^i_{D(\X)}(\O_{\mathbf{x}},\O_{\mathbf{y}}),
\end{split}
\end{equation}
which implies that $\F$ is strongly simple over $\X$.

Conversely, suppose that $\F$ is strongly simple over $\X$.  We show that this implies that $s^*(\F)$ is strongly simple over $X$.  To see this, we first show the orthogonality condition: let $\mathbf{x}$ and $\mathbf{y}$ denote geometric points of $\X$ and let $x$ and $y$ denote the corresponding geometric points on $X$, i.e. the image of $\mathbf{x}$ and $\mathbf{y}$ under the rigidification map.  $\Hom^i$ commutes with $s^*$ by Proposition~\ref{prop:sisflat}, so
\[
s^*(\Hom^i(\F_{\mathbf{x}},\F_{\mathbf{y}}))\cong \Hom^i(s^*(\F)_x,s^*(\F)_y).
\]
The map $s^*$ gives an equivalence of categories between $\Coh(\X)$ and sheaves on $X$ twisted by the appropriate Brauer class, which means that $s^*(\Hom^i(\F_{\mathbf{x}},\F_{\mathbf{y}}))=0$ if and only if $\Hom^i(\F_{\mathbf{x}},\F_{\mathbf{y}})=0$.  By the equality above, $s^*(\Hom^i(\F_{\mathbf{x}},\F_{\mathbf{y}}))=0$ if and only if $\Hom^i(s^*(\F)_x,s^*(\F)_y)=0$.  Combining these statements, we get that $\Hom^i(\F_{\mathbf{x}},\F_{\mathbf{y}})=0$ if and only if $\Hom^i(s^*(\F)_x,s^*(\F)_y)=0$.

Now we prove that if $\Hom^0(\F_{\mathbf{x}},\F_{\mathbf{y}})=k$, then $\Hom^0(s^*\F_x,s^*\F_y)=k$.  This follows since rank $1$ sheaves map to rank $1$ sheaves under the pullback $s^*$, so $s^*(\Hom^0(\F_{\mathbf{x}},\F_{\mathbf{y}}))=k$, and $s^*(\Hom^0(\F_{\mathbf{x}},\F_{\mathbf{y}})) \cong \Hom^0(s^*\F_x,s^*\F_y)$.  This shows that $s^*\F$ is strongly simple if $\F$ is.

We apply this fact to the following commutative diagram, where $t$ is any integer and $a$ is the weight of the sheaf $\F$.  The vertical arrows are given by $s^*$ on $\X$ and $\Y$, and both of the vertical maps are equivalences.  

\medskip
\begin{center}
\begin{tikzcd}
D^t(\X) \arrow[d, "\simeq"] \arrow[r, "\Phi^\F"] & D^{t+a}(\Y) \arrow[d, "\simeq"] \\
{D(X,t\cdot\alpha)} \arrow[r, "\Phi^{s^*(\F)}"]                & {D(Y,(t+a)\cdot\beta)}               
\end{tikzcd}
\end{center}
\medskip

By commutativity of the diagram, $\Phi^\F$ is an equivalence of categories if and only if $\Phi^{s^*\F}$ is.  

This proves that $\Phi^\F:D^t(\X) \to D^{t+a}(\Y)$ is fully faithful if $\F$ is strongly simple.  This holds for any $t$, and $D(\X)$ is the union of $D^t(\X)$ over all integers $t$.  To see that full faithfulness on each $D^t(\X)$ implies full faithfulness on all of $D(\X)$, note that every object and morphism in $D(\X)$ can be viewed as an object or morphism of some $D^t(\X)$.

\end{proof}

\section{Derived equivalences of gerbes}\label{sec:twisted}

We now consider the relationship between derived equivalences of gerbes and twisted derived equivalence of their coarse spaces, formalizing results in the proof of Proposition~\ref{prop:stackystronglysimple}. %

\begin{prop}
Let $\X$ and $\Y$ be $\G_m$-gerbes over $X$ and $Y$ respectively, corresponding to the Brauer classes $\alpha$ and $\beta$. Suppose $D(\X) \cong D(\Y)$ with kernel $\K \in D(\X \times \Y)$ where $\K$ has weight $1$.  Then $D(X,\alpha) \cong D(Y,\beta)$.
\end{prop}

\begin{proof}
    Consider the image $s^*(\K)$ of the kernel $\K$ under the forgetful map $s^*$ on the coarse space $X \times Y$.  Since $\K$ has weight $1$, $s^*(\K)$ is a twisted sheaf on $X \times Y$, with twisting given by $\pi_X^*(\alpha) \otimes \pi_Y^*(\beta)\in \Br(X \times Y)$.  Since $\K$ is strongly simple over $\X$, $s^*(\K)$ is strongly simple over $X$.  Therefore $s^*(\K)$ induces a fully faithful integral functor $\Phi^{r(\K)}:D(X,\alpha) \to D(Y,\beta)$.  Since $X$ and $Y$ are of the same dimension and are Calabi-Yau, any fully faithful integral transform is an equivalence by Proposition~\ref{prop:calabiyauequiv}, proving the statement.
\end{proof}

\begin{rem}
One could alternately prove the above statement by noting that $\Coh(X,\alpha)$ may be identified with a distinguished subcategory of $\Coh(\X)$, and likewise with $\Coh(Y,\beta)$ in $\Coh(\Y)$.  These distinguished subcategories are the ones consisting of all the sheaves on which $\gm$ acts via the standard multiplication action.  Since $\K$ has weight $1$ by assumption, one can note that the integral transform sends the distinguished weight $1$ subcategory of $D(\X)$ to the distinguished weight $1$ subcategory of $D(\Y)$.  
\end{rem}

The converse of this statement is not true, but it is true if one restricts to the distinguished subcategories $D^1(\X)$ and $D^1(\Y)$.
%{\color{blue} change this so that both $D(X)$ and $D(Y)$ are twisted; otherwise we're not proving this in full generality.}
\begin{thm}
\label{thm:twistedversusnormalderivedequivalences}
    Let $X$ and $Y$ be two Calabi-Yau schemes over a field $k$, and let $\alpha \in \Br(X)$ and $\beta \in \Br(Y)$ be Brauer classes such that
    $D(X,\alpha) \cong D(Y, \beta)$.
    Then there exists a derived equivalence
    \[
    D^1(\X) \cong D^1(\Y) 
    \]
    at the level of stacks, where $\X$ and $\Y$ are the gerbes associated to $\alpha$ and $\beta$. Further, if $ D(X,\alpha) \cong D(Y, \beta)$ is realized by an integral functor corresponding to the sheaf $F$, then the equivalence $D^1\X) \cong D^1(\Y)$ is realized by the integral functor corresponding to $f_1^*(F)$.
\end{thm}

\begin{proof}
   
     Consider the derived equivalence $D(X,\alpha) \cong D(Y,\beta)$. Suppose this is realized by the $\pi_X^*(\alpha) \otimes \pi_Y^*(\beta)$-twisted sheaf $F$.  
    Let $\F = f_1^{*}(F)$ be the image of $F$ under the weight $1$ section map. We now show that $D^1(\X) \cong D^1(\Y)$ is induced by $\Phi^{\F}$. This can be done by showing that $\F$ satisfies the following properties from Definition \ref{defn:stronglysimple}.
    
    We start by showing that $\Hom_{\Y}(\F_\mathbf{x},\F_\mathbf{x}) \cong k$ for all $\mathbf{x}\in \X$. Here, $\F_\mathbf{x}$ denotes the restriction $\F|_{\{\mathbf{x} \} \times \Y}$. Since the $\pi_X^*(\alpha) \otimes \pi_Y^{*}(\beta)$ twisted sheaf $F$ realizes the equivalence $D(X,\alpha) \cong D(Y, \beta)$, by strong simplicity for schemes, we know that $\Hom_Y(F_{\rig (x)}, F_{\rig(x)}) \cong k$. Since $F_{\rig(x)}$ is an $\alpha$-twisted sheaf on $Y$, by Proposition \ref{prop:equivofcats}, we have that $k \cong \Hom_{Y}(F_{\rig(x)},F_{\rig(x)}) \cong \Hom_{\Y}(f_1^{*}(F)_{\rig(x)}, f_1^{*}(F)_{\rig(x)}) = \Hom_{\Y}(\F_\mathbf{x}, \F_\mathbf{x})$.

    The statements that $\Hom^i(\F_{\mathbf{x}_1},\F_{\mathbf{x}_2})=0$ for all $i$ and all $\mathbf{x}_1 \neq \mathbf{x}_2 \in \X$, and that $\Hom^i(\F_{\mathbf{x}},\F_{\mathbf{x}}=0$ for $i<0$ and $i>\dim X$, follow by a similar argument, since we have an equivalence of categories between $\Coh(X,\alpha)$ and $\Coh^1(\X)$ from Proposition~\ref{prop:equivofcats}, and this equivalence preserves cohomological degree.

     Note that for any $\mathbf{x} \in \X$, $\F_\mathbf{x} \otimes \omega_{\Y} = \F_\mathbf{x}$, where $\omega_{\Y}$ is the canonical sheaf. This follows from the assumption that $\Y$ is Calabi-Yau. Thus $\F$ satisfies the condition for Proposition \ref{prop:calabiyauequiv}, and is indeed an equivalence.

    To finish off the proof, we need the fact that $\F \in \Coh^1(\X \times \Y)$. Note first that the gerbe $\X \times \Y \rightarrow X \times Y$ is respresented by the class $\pi_X^*(\alpha) \otimes \pi_{Y}^{*}(\beta) \in H^2_{et}(X \times Y, \G_m)$. By the equivalence of categories established in Section~\ref{sec:gerbeybasics}, $\F$ has a standard (weight $1$) $\G_m$ action.

\end{proof}

\section{Grothendieck-Riemann-Roch for gerbes}\label{sec:GRR}

In making explicit the study of derived equivalences of gerbey genus $1$ curves, it will be useful to have a version of Grothendieck-Riemann-Roch for $\gm$-gerbes. We will use equivariant Grothendieck-Riemann-Roch to deduce the corresponding statement for gerbes. Throughout this section, $X$ will denote a scheme of dimension $n$ and $G$ will denote a $g$-dimensional algebraic group acting on $X$. The stack $[X/G]$ will be denoted by $\X$ unless otherwise specified.

We recall the definitions of equivariant Chow groups and Grothendieck groups, as described in~\cite{edidingraham}. Let $V$ be an $\ell$-dimensional representation of $G$ such that there is an open set $U \subset V$ on which $G$ acts freely and such that $V\setminus U$ has codimension at least $n-i$ in $V$ (for some integer $i$).  We will call such a $V$ \emph{admissible for $G$}. In \cite{edidingraham}, the authors show that such a representation always exists. Let $U \to U/G$ be the principal bundle quotient (which is always at worst an algebraic space). The diagonal action of $G$ on $X \times U$ is free, so there is an algebraic space quotient $X \times U \to (X \times U)/G$ which is a principal $G$-bundle.  This quotient is denoted $X_G$.

\begin{defn}
Set $\CH_i^G(X)$, the $i$th equivariant Chow group of $X$, to be $\CH_{i+\ell-g}(X_G)$, for $\CH_*$ the usual Chow group for schemes.  Let $\CH_i^G(X)_\Q:=\CH_i^G(X)\otimes \Q$.
\end{defn}

It turns out that the equivariant Chow groups are independent of the representation $V$ as long as $V\setminus U$ has large enough codimension in $V$.  For a proof of this fact, see \cite{edidingraham}. Further, there is a natural graded ring structure on $CH_*^G(X)$; for details, see~\cite{EGRR}. If $Y \subset X$ is an $m$-dimensional $G$-invariant subvariety, then it has a $G$-equivariant fundamental class $[Y]_G\in Ch_m^G(X)$.  More generally, if $V$ is an $\ell$-dimensional $G$-representation and $S \subset X \times V$ is an $(m+\ell)$-dimensional subvariety, then $S$ has a $G$-equivariant fundamental class $[S]_G\in \CH_m^G(X)$.  This means that $\CH_i^G(X)$ can be nonzero for all $i\le \dim X$, including $i<0$. Since the equivariant Chow groups $\CH_i^G(X)$ are defined as ordinary Chow groups of the variety $X_G$, they enjoy all the functoriality properties of ordinary Chow groups.  One may also define $K_0(G,X)$, which is the Grothendick group of $G$-equivariant vector bundles on $X$.  The equivariant groups $K_0(G,X)$ and $\CH^G_{*}(X)$ satisfy Grothendieck-Riemann-Roch, as stated in  %\textcolor{red}{Okay, I was looking up Edidin-Graham and I was confused about something. In theorem 3.1 in the paper, which is where they state that $\tau^G$ is the equivariant RR map, etc., part (c) says that if $\epsilon \in K_0^G$ and $\alpha \in G^G$, then $\tau^G(\epsilon \alpha) = ch^G(\epsilon)\tau^G(\alpha)$. In particular (as a sanity check) $\tau^G(1)$ should be $Td(TX)$. See Fulton Chapter 18 for reference. Also, the paper doesn't state the version that we state below exactly, so I put in some modifying sentences.} 
the following theorem, which is a consequence of Theorem 3.1 in \cite{EGRR}. %\textcolor{red}{Not strictly important, but why does $\tau_X$ land only in the positive part of $\CH_{*}$?}

\begin{notation}
For a vector bundle $E$ on $X$, we denote by $ch(E)$ and $Td(E)$, the Chern character and Todd class of $E$ respectively. If $E \in K_0(G,X)$, then its $G$-equivariant Chern character and Todd class will be denoted by $ch^G(E)$ and $Td^G(E)$ respectively. We refer the reader to \cite{EGRR} for the precise definitions of the equivariant versions: essentially, these may be calculated by calculating the classical $ch$ and $Td$ on the space $X_G$.
\end{notation}

\begin{thm}[]\label{thm:equivariantRR}
There exists an equivariant homomorphism $\tau_X:K_0(G,X) \to \underset{i\ge 0}{\prod}\CH^G_{i}(X) \otimes \Q$. 
The map $\tau_X$ is covariant for proper $G$-equivariant morphisms and when $X$ is a smooth scheme and $\varepsilon \in K_0(G,X)$ is the class of a vector bundle, then %(\textcolor{red}{modulo showing $\tau(1) = Td(TX-\mathfrak{g})$})
\[
\tau_X(\varepsilon)= \frac{ch(\varepsilon)Td(TX-\mathfrak{g})}{Td^G(V)}
\]
where $\mathfrak{g}$ is the adjoint representation of $G$ and $V$ is a choice of $G$-representation that is admissible for $G$.  Furthermore, the map $\tau_X$ is the map which makes Grothendieck-Riemann-Roch hold.  In particular, the following diagram is commutative:
\begin{center}
\begin{tikzcd}
K_0^G(X\times Y) \arrow[r, "\pi_!"] \arrow[d, "\tau_{X\times Y}"] & K_0^G(X) \arrow[d, "\tau_X"] \\
\CH_*^G(X \times Y)_\Q \arrow[r, "\pi_*"]                         & \CH_*^G(X)_\Q.                        
\end{tikzcd}
\end{center}

\end{thm}
%\textcolor{red}{Notation question: Is $K_0^G(X)$ the same as $K_0(G,X)$? Also should the $\CH^i_{G}(X)$ be $\CH^G_{i}(X)$? Oh wait, Edidin-Graham uses lower numbering sometimes - ugh, confusing.}

Indeed, $\tau_X$ is typically defined on the larger group $G_0(G,X)$, which is the Grothendieck group of equivariant coherent sheaves on $X$. $K_0(G,X)$ is a subgroup of $G_0(G,X)$, and for our purposes, we will only need the restriction of $\tau_X$ to $K_0(G,X)$. Theorem 3.1 in \cite{EGRR} states that if $\epsilon \in K_0(G,X)$ and $\alpha \in G_0(G,X)$ then $\tau_X^G(\epsilon \alpha) = ch^G(\epsilon) \tau_X^G(\alpha)$. Choose $V$ and $U$ as in the definition of $\CH^G_{*}(X)$. Then $\tau_X(\alpha)$ is defined as the element of $\prod \CH^G_j(X)_{\Q}$ whose projection to the $j$-th component agrees with that of
$$
\frac{\tau_{X_G}(\alpha)}{Td^G(V)},
$$
where $\tau_{X_G}$ agrees with the classical GRR map. %(\textcolor{red}{Here onwards is where we are claiming things}) 
Now, if $\alpha = [\mathcal{O}_X]$ in $G_0(G,X)$, then 
$$
\tau_{X_G}(\alpha) = Td(T_{X \times U/ G}) = Td(TX - \mathfrak{g}),
$$
since $G$ acts trivially on $U$. This shows how Theorem \ref{thm:equivariantRR} follows from Theorem 3.1 in \cite{EGRR}. Further, in \cite{EGRR}, the authors show that the equivariant GRR map is independent of the choice of $U$ and $V$. In particular, if $G=GL_n$, we may take $V$ corresponding to trivial $G$-bundle over $X$, i.e. the usual matrix representation of $GL_n$. In this case, we have that for $\epsilon \in K_0(G,X),$ $\tau_X(\epsilon) = ch(\epsilon)Td(TX - \mathfrak{g})$.
%letting $V$ be the trivial $G$-bundle over $X$, we have that for $\epsilon \in K_0(G,X),$ $\tau_X(\epsilon) = ch(\epsilon)Td(TX - \mathfrak{g})$. 
%\textcolor{red}{I'm worried about taking the trivial bundle here. When EG show the existence of an admissible representation $G$, they first embed it into $GL_n$ and then show that any case can be reduced to that of $GL_n$. And then for $GL_n$ they take the trivial bundle. Why not take the trivial $G$ bundle directly? Oh, I guess they need a linear action, so maybe one has to show that a ``trivial" $G$-bundle does indeed give a vector space $V$. I think this works, but I'm also happy to punt on it and just use say that for $G=GL_n$, we can take the trivial bundle, i.e. $GL_n$ acting on matrices in the usual way, an leave it at that.}
%\textcolor{blue}{Suggested edit: In particular, if $G=GL_n$, we may take $V$ corresponding to trivial $G$-bundle over $X$, i.e. the usual matrix representation of $GL_n$. In this case, we have that for $\epsilon \in K_0(G,X),$ $\tau_X(\epsilon) = ch(\epsilon)Td(TX - \mathfrak{g})$.}

\begin{rem}
In the case that $G$ is solvable, $\mathfrak{g}$ is a trivial representation of $G$ and the formula may be written as $\tau_X(V)=ch(V)Td(TX)$.
\end{rem}

We now turn to defining Chow rings and Grothendieck rings on gerbes, and show that for quotient stacks, these groups coincide with their equivariant avatars on schemes.   %\textcolor{red}{Let $\X$ be any algebraic stack. The Chow ring of $\X$ may be defined as follows.  An element $c\in \CH^k(\X)$ defines an operational $\CH_*(B) \overset{c_f}{\to} \CH_{*-k}(B)$ for any scheme-valued point $f:B \to \X$ of $\X$.  The operationals are compatible with proper pushforward, flat pullback, and intersection products for all scheme-valued points of $\X$, which gives $\CH_*(\X)$ a ring structure.  It turns out that when $\X$ is a quotient stack, the Chow ring of $\X$ coincides with the equivariant Chow ring of $X$} \textcolor{blue}{Why do we need this? Are we writing this to show that $\CH$ of a stack has a ring structure? If so we should either be concrete about what $c_f$ is or just state that it has a ring structure and give a reference. The current form doesn't really help readers who don't know about this and isn't required people who know about this either. Also, if you're okay with it, I think this should go earlier with the definition of Chow groups.}

\begin{prop}
Let $[X/G]$ be a smooth quotient stack.  Then $\CH^*([X/G])\cong \CH^*_G(X)$.  Furthermore, the integral Chow groups of $[X/G]$ are $\CH_i([X/G]) = \CH^G_{i-g}(X)$ for $g=\dim G$.
\end{prop}

For a proof of this proposition, see ~\cite{edidingraham}. Next, we note that the Grothendieck group of vector bundles on a quotient stack has an explicit equivariant description.

%We now discuss the equivariant Grothendieck group of vector bundles on a scheme.  The group $K_0(G,X)$ is defined to be the Grothendieck group of $G$-equivariant vector bundles on $X$, and $G_0(G,X)$ is defined to be the Grothendieck group of $G$-equivariant coherent sheaves on $X$.  Similarly, $K_0(\X)$ and $G_0(\X)$ are the Grothendieck groups of vector bundles on $\X$ and coherent sheaves on $\X$, respectively.

\begin{prop}
Let $G$ be an algebraic group acting on the variety $X$. Then, $K_0([X/G]) \cong K_0(G,X)$.
\end{prop}

\begin{proof}
This holds because a vector bundle on $[X/G]$ is simply the data of a $G$-equivariant vector bundle on $X$.
\end{proof}

The above propositions give a clean description of $\CH_{*}(\X)$ and $K_0(\X)$ in terms of their equivariant versions when $\X$ is a quotient stack. For an arbitrary stack, such relations need not hold. However, in the case that $\X$ is a $\gm$-gerbe over a genus $1$ curve, it can be realized as a global quotient stack. In \cite{ehkv}, Edidin, Hassett, Kresch and Vistoli give the following criterion for a $\gm$-gerbe to be a global quotient:

\begin{prop}\label{prop:quotientstack}[\cite{ehkv}, Theorem 3.6]
Let $X$ be a noetherian scheme and let $\beta \in H^2_{\et}(X,\gm)$.  Further let the Brauer map denote the injective map from the Azumaya Brauer group to the cohomological Brauer group.  Then the following are equivalent:

(i) $\beta$ lies in the image of the Brauer map.

(ii) the $\gm$-gerbe classifying $\beta$ is a quotient stack. %\textcolor{red}{is it always a quotient by an algebraic group? I would think yes, since I think whatever the group is has to be a subgroup of $GL_n$ for some $n$?}  \textcolor{blue}{Yes, it is always a quotient by $GL_n$ for some $n$.  The theorem in~\cite{ehkv} proves that the quotient stack structure is always of the form $[X \times PGL_r/GL_r]$ for some $r$.}
\end{prop}

By a theorem of Gabber (proved in \cite{deJong}), if $X$ is a quasicompact separated scheme admitting an ample invertible sheaf, then the Azumaya Brauer group is the same as the cohomological Brauer group. Thus if $\X$ is a $\gm$-gerbe over such a scheme $X$, then $\X$ arises as a global quotient stack. In fact, if $\X$ corresponds to $\beta \in H^2_{\et}(X, \gm)$, then $\X = [\widetilde{X}/ GL_n]$, where $\widetilde{X}$ is the $PGL_n$ torsor over $X$ corresponding to the preimage of $\beta$ under the boundary map $H^1_{\et}(X, PGL_n) \to H^2_{\et}(X, \gm).$

Combining Proposition \ref{prop:quotientstack} with equivariant Grothendieck-Riemann-Roch (Theorem \ref{thm:equivariantRR}) gives us the following.

\begin{thm}\label{thm:GRR}
Let $\X$ be a $\gm$-gerbe over a quasicompact separated scheme $X$, with $X$ admitting an ample line bundle.  Let $\Y$ and $Y$ be defined similarly. Then there is a commutative diagram
\begin{center}
\begin{tikzcd}
K_0(\X\times \Y) \arrow[r, "\pi_!"] \arrow[d, "\tau_{\X \times \Y}"] & K_0(\X) \arrow[d, "\tau_\X"] \\
\CH_*(\X \times \Y)_\Q \arrow[r, "\pi_*"]                         & \CH_*(\X)_\Q.                         
\end{tikzcd}
\end{center}
When the input is a vector bundle, $\tau_{\X \times \Y}(-)=ch(-)\cdot Td(T_{\X \times \Y})$, and similarly for $\tau_\X$.
%where $K_0^{\gm}$ and $\CH_*^{\gm}$ denote $\gm$-equivariant $K$-groups and Chow groups, respectively.
\end{thm}

\begin{proof}
By Proposition \ref{prop:quotientstack}, we may write $\X = [\widetilde{X}/GL_{n_1}]$ and $\Y = [\widetilde{Y}/GL_{n_2}]$ for some integers $n_1$ and $n_2$.  %(\textcolor{red}{These can always be taken to be $GL_n$ and $GL_m$, right?}) \textcolor{red}{Yep!}. 
Define a modified action of $GL_{n_1}\times GL_{n_2}$ on $\widetilde{X}$ by $(g_1,g_2) \cdot x = g_1 \cdot x$ for $x \in \widetilde{X}$. Thus, setting $G = GL_{n_1} \times GL_{n_2}$, we have that
\begin{align*}
    K_0(\X) \cong K_0(GL_{n_1},\widetilde{X}) \cong K_0(G,\widetilde{X}),\\
    \CH_{*}(\X) = \CH_{*}^{GL_{n_1}}(\widetilde{X}) = \CH_{*}^{G}(\widetilde{X}).
\end{align*}

Now, $\X \times \Y = [\widetilde{X} \times \widetilde{Y}/GL_{n_1} \times GL_{n_2}]$. The theorem then follows from Theorem \ref{thm:equivariantRR}.
\end{proof}

%%Any smooth genus $1$ curve $C$ over $k$ satisfies the hypotheses of Theorem~\ref{thm:quotientstack}, which means that any $\gm$-gerbe $\Cc$ over $C$ is a global quotient stack.  It turns out that the quotient structure can be described in an explicit way.  By Theorem~\ref{thm:quotientstack}, for each Brauer class $\alpha$ in $H^2_{\et}(C,\gm)$ there is a class $\beta \in H^1_{\et}(C,PGL_n)$ for some $n$ such that $\alpha$ is the image of $\beta$ under the Brauer map $H^1_{\et}(C,PGL_n) \to H^2_{\et}(C,\gm)$.  Therefore $\alpha$ corresponds to a $PGL_n$ bundle over $C$; call it $\widetilde{C}$.  Now the gerbe $\Cc=[\widetilde{C}/GL_n]$.  This presentation of $\Cc$ lets us compute its intersection theory explicitly; this will be put to use in Section~\ref{sec:lowdim}.

%(\textcolor{red}{Might need relocating})When the projection map in Theorem~\ref{thm:GRR} is of the form $\pi:\Cc \times \Cc' \to \Cc$, for $\Cc$ and $\Cc'$ gerbey genus $1$ curves, with $\Cc=[\widetilde{C}/GL_n]$ and $\Cc'=[\widetilde{C'}/GL_m]$, the $G$ in the theorem statement is $GL_n \oplus GL_m$.  $GL_n \oplus GL_m$ acts on the left hand side via the quotient stack structure.  On the right hand side, the $GL_m$ summand acts trivially on $\Cc$ and $GL_n$ acts via the quotient stack structure.

It is worth noting that there has been a good deal of other work on Riemann-Roch type theorems for stacks.  A version of Grothendieck-Riemann-Roch for $\mu_n$-gerbes is proved in ~\cite{maxrr}, and ~\cite{toen} proves various Riemann-Roch theorems for general DM stacks.

The following statement of Riemann-Roch for gerbey curves is a corollary of Theorem~\ref{thm:GRR} obtained by setting the gerbes to by $\gm$-gerbes over smooth proper curves and letting the vector bundles in the statement be line bundles.

\begin{cor}[Riemann-Roch for gerbey curves]
\label{cor:gerbeyRRforcurves}
Let $\Cc \rightarrow C$ be a gerbey curve over a smooth curve of genus $g$. Let $\L$ be a line bundle on $\Cc$. Then
$$
h^0(\Cc, \L) - h^0(\Cc, \L^{\vee} \otimes \omega_{\Cc}) = \deg(\L) - g + 1.
$$
In particular, if $g=1$ and $\deg(\L)>0$, then $h^0(\Cc, \L) = \deg(\L)$.
\end{cor}

\section{Applications to low dimensions}\label{sec:lowdim}

We now apply the ideas developed in the preceding sections to consider in detail the case of derived equivalences of gerbey genus 1 curves. Genus 1 curves are not Calabi-Yau, since $h^1(X,\O_X)=1$.  Nevertheless, their canonical bundles are trivial, and hence by \cite{bondalorlov}, they may still admit nontrivial derived equivalences.

\subsection{Gerbey genus $1$ curves}
\label{subsec:gerbeygenus1}
The goal of this section is to generalize Theorems \ref{lem:akw} and \ref{thm:akw} proved by Antieau, Krashen and Ward. These generalizations are non-trivial, since even if $C' = Pic^n(C)$ for some non-zero integer $n$, $C \times C'$ need not admit a universal sheaf.

\begin{ex}
\label{example:nouniversalsheaf}
 Let $k=\Q$ and let $E/k$ be an elliptic curve with $j(E) \neq 0, 1728$. In this case $\Aut_k(E) \cong \Z/2\Z$. Let $C \in H^1(k,E)$ be element of order $n=6$. For each $d \in \Z$, by (\cite{AKW}, Lemma 2.7) we have that $dC$ is represented by the homogeneous space $Pic^d(C)$ of degree $d$ line bundles on $C$. Let $C'=Pic^2(C)$. Since $\phi$ is of order 2, by Theorem \ref{thm:akw}, it suffices to show that $\phi_{*}C' = dC$ where $d$ is not coprime to 6. Now, the automorphism $\phi: E \rightarrow E$ sends a point to its negative and thus acts as $[-1]$ on $H^1(k,E)$. Thus $\phi_{*}(C') = 4C$, and in particular $C \times C'$ does not admit a derived equivalence.
 %in this case H^1 is large enough to find in-equivalent torsors
 %j \neq 0,1728 allows us to restrict the number of automorphisms, i.e. the size of Aut_k(E)
\end{ex}

\begin{ex}
In ~\cite{AKW}, the authors prove that over $\R$, $D(X) \cong D(Y) \implies X \cong Y$.  Therefore, let $C/\R$ be a curve without an $\R$-point and let $C'$ be the coarse space of its Jacobian.  Now $C$ and $C'$ are not isomorphic, and therefore cannot be derived equivalent.  In particular, there cannot exist a universal sheaf, since Orlov's theorem guarantees that if $D(C) \cong D(C')$, then $C \times C'$ must admit a universal sheaf. %a universal sheaf would give the kernel of an equivalence.  
%\textcolor{red}{AKW is filled with examples of fields over which $D(X) \cong D(Y) \implies X \cong Y$. For instance finite fields, or $\mathbb{R}$. So if we start with two non-isomorphic curves over these fields - for instance a curve without a rational point and its jacobian - we can get curves that are not derived equivalent. In particular, there cannot exist a universal sheaf in this case (maybe have to use Orlov representability). I will let you pick your favorite example}
\end{ex}

\begin{rem}
Antieau, Krashen and Ward say that a scheme $M$ is a \emph{fine moduli space} (of sheaves) for a moduli problem on a scheme $X$ if it admits a universal sheaf. We use the term ``fine moduli space'' to mean the \emph{stack} $\mathcal{M}$ admitting a universal sheaf for the moduli problem. That is, our ``fine" moduli spaces admit a universal sheaf and encode all the automorphisms of the sheaves being parameterized.  Therefore in this paper, all fine moduli spaces are $\gm$-gerbes, whereas in~\cite{AKW} the fine moduli spaces being used are schemes.
\end{rem}

By studying fine moduli spaces, we will obtain examples of derived equivalent gerbey curves which cannot be obtained using the results in ~\cite{AKW}. We begin with some definitions.

\begin{defn}
\label{defn:gerbeygenus1}
A \emph{gerbey genus 1 curve} $\Cc$ over $k$ is a $\gm$-gerbe over a genus $1$ curve $C$ over $k$. Here, $C$ is the coarse space for $\Cc$.
\end{defn}

Note that while $\Cc$ is called a gerbey \emph{curve}, it is a ($1-1=0$) dimensional stack. The key example of a gerbey genus $1$ curve is a connected component of the Picard stack $\Pic(C)$ for $C$ an ordinary genus $1$ curve. Any line bundle on a curve has a $\gm$'s worth of automorphisms, so any component of the Picard stack is a $\gm$-gerbe over the corresponding component of the Picard scheme.  

\begin{defn}
\label{defn:weight1pic}
Define $\Pic^d_1(\Cc)$ to be the moduli space of weight $1$ degree $d$ line bundles on $\Cc$, where weight $1$ refers to the $\gm$-action in the sense of Section 2.
\end{defn}

The substack $\Pic^d_1(\Cc)$ of $\Pic^d(\Cc)$ is a gerbey genus $1$ curve, as we prove below, in a variation of ~\cite[Prop. 1.1.7]{dannymax}.

%In the case of curves or K3 surfaces, one finds Fourier-Mukai partners of the variety by using components of the moduli space of stable sheaves which have the same dimension as the original variety.  In order to find Fourier-Mukai partners of a gerbey genus $1$ curve, then, we wish to find components of the moduli space of sheaves that are of dimension $(1-1)$, that is, with $1$-dimensional coarse space and $1$-dimensional automorphism groups.  The fact that $\Pic^d_1(\Cc)$ is a gerbey genus $1$ curve is proved in the following theorem.

\begin{thm}\label{thm:picofgerbes}

Let $\Cc\to C$ be a $\G_m$-gerbe with class $\alpha \in H^2(C,\gm)$.  Then $\Pic^d_1(\Cc) \cong \Pic^d(C,\alpha)$ for any integer $d$.  

\end{thm}
%\textcolor{red}{I don't believe we need to prove most of this any more - the current statement is a corollary to 2.2 and is true for any variety, right? So we can just state it as such, I think.} 

\begin{proof}

The local section map $s^*:\Coh(C,\alpha) \to \Coh^1(\Cc)$ between weight $1$ sheaves on $\Cc$ and $\alpha$-twisted sheaves on $C$ preserves rank, since it is given by pullback along a flat base change.  Therefore this map sends line bundles to line bundles. Therefore by Proposition~\ref{prop:equivofcats}, $\Pic^d_1(\Cc)\cong \Pic^d(C,\alpha)$ as abelian categories. Since $s^*$ preserves families, this is additionally an isomorphism of moduli stacks.
\end{proof}

Note that if we had chosen to parameterize weight $t$ degree $d$ line bundles on $\Cc$ for some $t\neq 1$, we could have played a similar game, using the fact that $\Pic^d_t(\Cc)$ is equivalent to $\Pic^d(C,t\cdot \alpha)$.

\begin{cor}
If $\Cc$ is a gerbey genus $1$ curve, then for any $d$, $\Pic^d_1(\Cc)$ is also a gerbey genus $1$ curve.
\end{cor}

\begin{proof}
This follows immediately from Theorem~\ref{thm:picofgerbes} and the fact that $\Pic^d(C)$ is a gerbey genus $1$ curve.
\end{proof}

We make a statement for gerbey genus $1$ curves analogous to Theorem~\ref{thm:akw}.%

\begin{thm}%\label{thm:gerbeygenus1}
Let $\Cc$ be a gerbey genus $1$ curve over $k$, and let $\Cc':=\Pic^d_1(\Cc)$ be the fine moduli space of weight $1$ degree $d$ line bundles on $\Cc$ for $d\neq 0$.  Let $\L$ be the universal weight $1$ degree $d$ line bundle on $\Cc \times \Cc'$.  The integral functor
\[
\Phi^\L:D(\Cc) \to D(\Cc')
\]
is an equivalence. 
\end{thm}

\begin{proof}

By Proposition \ref{prop:stackystronglysimple}, it suffices to check that $\L$ is strongly simple over $\Cc$.  We will show that the universal weight $1$ degree $d$ line bundle $\L$ is strongly simple if $d>0$. 

% \textcolor{red}{We never changed the paragraph below when we changed the strong simplicity section. Did that below.} \textcolor{blue}{Thanks, looks great!}
To check strong simplicity, we need to check that
\begin{itemize}
    \item $\Hom^0(\L_{\mathbf{x}},\L_{\mathbf{y}})=k$ as sheaves on $\Cc$ for any $\mathbf{x}, \mathbf{y} \in \Cc'$ with $\mathbf{x} \sim \mathbf{y},$
    \item $\Hom^i(\L_{\mathbf{x}},\L_{\mathbf{y}}) = 0$ when $i<0$ or $i>1$, and $\mathbf{x} \sim \mathbf{y},$ and
    \item $\Hom^i(\L_{\mathbf{x}},\L_{\mathbf{y}})=0$ for all $i$ if $\mathbf{x} \not\sim \mathbf{y}$. 
\end{itemize}

 By Propositions \ref{prop:equivofcats} and \ref{prop:sisflat}, it suffices to check that these $\Hom$ sheaves satisfy the desired properties over the coarse space $\Cc \to C$. Recall that the \'etale and Zariski cohomology of quasicoherent sheaves coincides (see ~\cite[03P2]{stacks}).  Therefore Riemann-Roch applies in our setting. Let $x$ and $y$ denote images of $\mathbf{x}$ and $\mathbf{y}$ respectively under the coarse space map $\Cc' \to C'$. Note that by definition, if $\mathbf{x} \sim \mathbf{y},$ then $x=y$. By abuse of notation, we use $\L_x$ and $\L_y$ to mean the sheaves $s^{*}\L_{\mathbf{x}}$ and $s^{*}\L_{\mathbf{y}}$ on the coarse space $C$. We check first that $\Hom^0(\L_x,\L_x)=k$.  We have that
\begin{align*}
    \Hom^0(\L_x,\L_x)&=H^0(C,\L_x \otimes \L_x^\vee) \\
 &=H^0(C,\O_{C})=k.
\end{align*}

To show the other two statements, observe that
\[
\Hom^i(\L_x,\L_y)=H^i(C,\L_y \otimes \L_x^\vee).
\]
Thus when $x=y$, $\Hom^i(\L_x,\L_y) = H^i(C,\O_{C}),$ which is 0 for $i<0$ or $i>1$. When $x\neq y$, $\L_y \otimes \L_x^\vee$ is a nontrivial degree $0$ line bundle. The nontriviality implies that it has no global sections, therefore $\Hom^0(\L_x,\L_y)=0$.  By Serre duality and trivality of the canonical bundle, $\Hom^1(\L_x,\L_y)=0$.  

\end{proof}

Note that when $d\neq 0$, $\Phi^{\L^\vee}$ as a functor $D(\Cc') \to D(\Cc)$ inverts the derived equivalence in the above theorem, since the dual of a strongly simple sheaf is strongly simple. The following corollary makes it clearer how the above theorem relates to the Picard stacks of ordinary genus $1$ curves.

\begin{cor}
Let $C$ be a genus 1 curve over $k$, and $\Cc':=\Pic^d(C)$ be the fine moduli space for $d\neq 0$.  Let $\L$ be the pullback of universal degree $d$ line bundle on $C \times \Cc'$ to $C \times \bgm \times \Cc'$; then the integral functor $\Phi^\L:D(C\times B\G_m) \to D(\Cc')$ is an equivalence if $d\neq 0$. %, and the integral functor $\Phi^{\L^\vee}:D(C \times \bgm) \to D(\Cc')$ is an equivalence if $d<0$.
\end{cor}

In particular, if $\alpha$ is the class corresponding to the gerbe $\Cc' \rightarrow C'$ over the coarse moduli space $C'$, then we get a derived equivalence of twisted categories, $D(C) \xrightarrow{\sim} D(C', \alpha).$

\begin{proof}
Take $\Cc$ in Theorem~\ref{thm:gerbeygenus1} to be the trivial gerbe $C \times \bgm$.  Now use the fact that sheaves on $C$ are the same thing as weight $1$ sheaves on $C \times \bgm$, and the result follows.
\end{proof}

\subsection{Some moduli-theoretic computations}

The goal of this section to understand is how to make the derived equivalences in Section \ref{subsec:gerbeygenus1} more explicit. For instance, if $\Cc'$ is a Fourier-Mukai partner of a curve $\Cc$ obtained by taking the moduli of line bundles of degree $d$, how do we produce $\Cc$ as the moduli of vector bundles of some rank and degree on $\Cc'$?  Or, how do we compose such Fourier-Mukai transforms in an explicit moduli-theoretic way?  %\textcolor{red}{should all of these be mathscroll?} {\color{violet} yes, they should.  this is fixed now.}

In answering these questions, we will frequently use the following result of Atiyah from~\cite[Theorem 7]{atiyah}.
%{\textcolor{red}{I didn't find this theorem the last time I looked for it - need to look again.}}

\begin{thm}\label{thm:atiyah}
If $r,d$ are coprime positive integers, and $J_E(r,d)$ denotes the moduli of rank $r$ degree $d$ stable vector bundles on an elliptic curve $E$, then the determinant map $\det:J_E(r,d) \to J_E(1,d)$ is an isomorphism.
\end{thm}

Atiyah assumes that the base field is algebraically closed.  Therefore we may assume only that whenever $r$ and $d$ are coprime, the determinant map $\det:J_E(r,d) \to J_E(1,d)$ is \emph{geometrically} an isomorphism. Before we prove Theorem~\ref{thm:explicitdim1}, we prove some intermediate lemmas.

\begin{lem}
\label{lem:toddtrivial}
Let $\Cc \to C$ and $\Cc' \to C'$ be two gerbey genus 1 curves over smooth curves $C$ and $C'$. Then the maps
$$
Td(\Cc \times \Cc') : K_0(\Cc \times \Cc') \rightarrow \CH_{*}(\Cc \times \Cc')
$$
and 
$$
Td(\Cc') : K_0(\Cc') \rightarrow  \CH_{*}(\Cc')
$$
from Theorem \ref{thm:GRR} are both trivial.  
\end{lem}

\begin{proof}
Let $\pi_1$ and $\pi_2$ denote the two projections $\Cc \times \Cc' \to \Cc$ and $\Cc \times \Cc' \to \Cc'$ respectively. The tangent bundle $T_{\Cc \times \Cc'}=\pi_1^*T_{\Cc} \oplus \pi_2^*T_{\Cc'}$.  Since $Td$ is a homomorphism from $K_0(\X) \to \CH_*(\X)$ for any $\X$, we have that $Td(T_{\Cc \times \Cc'})=Td(\pi_1^{*}T_{\Cc})\cdot Td(\pi_2^{*}T_{\Cc'}) = \pi_1^{*}Td(T_{\Cc}) \cdot \pi_2^{*}Td(T_{\Cc'})$. Thus it suffices to show that $Td(T_{\Cc}) = Td(T_{\Cc'}) =1$. Since $\Cc$ and $\Cc'$ are interchangeable, we show below that $Td(T_{\Cc}) = 1.$

%\textcolor{red}{Why is $Td(\pi_1^{*}T_{\Cc}) = \pi_1^{*}Td(T_{\Cc})$?} \textcolor{blue}{Where do we need it to be?  This should be true because blah blah functoriality of the Todd class, but I don't see where we use that} \textcolor{red}{Well, because $Td(T_{\Cc \times \Cc'}) = Td(T_{\Cc}) \cdot Td(T_{\Cc'})$ doesn't make sense.} 
%{\textcolor{red}{Since $\tau$ is written multiplicatively, I think $Td$ should be written multiplicatively too. It's not a big deal, but it's better writing-wise. I can go through and change it if you're okay with that.}}\textcolor{blue}{Yep, that's fine with me!}\\

Write $\Cc=[\widetilde{C}/GL_n]$ as a quotient stack, for $\widetilde{C}$ the $PGL_n$ bundle on $C$ arising from the interpretation of $\alpha$ as a class in $H^1(C,PGL_n)$. Then $T_{\Cc}=T\widetilde{C}-\mathfrak{gl}_n$, where $\mathfrak{gl}_n$ is the adjoint representation of the stabilizer $GL_n$.  The morphism $\widetilde{C} \to C$ given by the $PGL_n$ bundle map gives a short exact sequence
\[
0 \to T_{\widetilde{C}/C} \to T\widetilde{C} \to TC \to 0.
\]
Now $TC=\O_C$ since $C$ has genus $1$, so $T\widetilde{C}$ is determined by the relative tangent bundle, which is $\mathfrak{pgl}_n$.  Therefore $T\Cc=\mathfrak{pgl}_n-\mathfrak{gl}_n$.  $Td:K_0(\Cc) \to \CH_*(\Cc)$ is a homomorphism, so $Td(\Cc)=Td(C)+Td(\mathfrak{pgl}_n-\mathfrak{gl}_n)=Td(C)+Td(\mathfrak{g}_m)$, where $\mathfrak{g}_m$ is the adjoint representation of $\gm$. %, so that $T\Cc=\mathfrak{g}_m$, where $\mathfrak{g}_m$ is the adjoint representation of $\gm$.
Since $\gm$ is abelian, $\mathfrak{g}_m$ is trivial, which means that $Td(\mathfrak{g}_m)=1$.  Since $TC$ is trivial, this means that $Td(\Cc)=1$.
\end{proof}

\begin{lem}
\label{lem:simplifiedgrrforcurves}

Let $\Cc$ and $\Cc'$ be as in Lemma \ref{lem:toddtrivial} and let $\L$ be a line bundle on $\Cc \times \Cc'$ of degree $d>0$. Let $\pi: \Cc \times \Cc' \to \Cc'$ denote the projection to the second component. Recall the commutative diagram given to us by Theorem \ref{thm:GRR}.

\begin{center}
\begin{tikzcd}
K_0(\Cc\times \Cc') \arrow[r, "\pi_!"] \arrow[d, "ch(\cdot)Td"] & K_0(\Cc') \arrow[d, "ch(\cdot)Td"] \\
\CH_*(\Cc \times \Cc') \arrow[r, "\pi_*"]                         & \CH_*(\Cc')                         
\end{tikzcd}
\end{center}

%This holds since these are the Todd classes of the tangent bundle of $C'$ and $C \times C'$.  Since $C'$ has genus 1, it has trivial cotangent bundle, therefore trivial tangent bundle, and a trivial bundle has Todd class 1.  Since $T(C \times C')=\pi^*(T(C)) \oplus \pi^*(T(C'))$, $T(C \times C')$ also has trivial Todd class.  
Then, $ch(\pi_*\L)=\pi_*(ch(\L))$. 
\end{lem}

\begin{proof}
First, note that $\pi_!:K_0(\Cc \times \Cc') \to K_0(\Cc')$ is just the ordinary pushforward.  This holds since $R^1$ of any fiber is the first cohomology of a degree $d$ line bundle on $\Cc$, which is Serre dual to the global sections of a degree $-d$ line bundle on $\Cc$.  Since $d>0$, a degree $-d$ line bundle has no sections. The pushforwards $R^i$ for $i\ge 2$ vanish for dimension reasons. This allows us to replace the $\pi_!$ in the diagram with $\pi_*$.
\end{proof}

We now prove Theorem \ref{thm:explicitdim1}, which we recall here for convenience.

\begin{thm}
\label{thm:explicitdim1restated}
Let $\Cc'=\Pic^d_1(\Cc)$ for some $d>0$ such that $\Cc$ and $\Cc'$ are derived equivalent.  Then $\Cc$ is the moduli space of weight $1$, rank $d$ degree $1-d$ vector bundles on $\Cc'$, and the universal sheaf for this moduli problem is the dual of the universal degree $d$ line bundle $\L$, whose associated integral functor inverts the derived equivalence $\Phi^\L:D(\Cc) \to D(\Cc')$.  
\end{thm}

\begin{proof}

Let $\pi: \Cc \times \Cc' \rightarrow \Cc'$ be the projection map to the second component. We begin by calculating the rank and degree of $\pi_{*} \L$ on $\Cc'$. By gerbey Grothendieck-Riemann-Roch (Theorem~\ref{thm:GRR}), we may compute the rank and degree using the commutative diagram given by GRR. %Proposition~\ref{prop:sisflat} and the fact that rank and degree of vector bundles are preserved under flat base change, we can run the intersection theory computation for these invariants on the coarse spaces $C$ and $C'$, rather than working directly on the gerbe.

We have assumed that $d>0$.  By Riemann-Roch on $\Cc'$, $\pi_*\L$ on $\Cc'$ has $d$ many sections.  To see this, consider some point $x\in \Cc'$. The restriction of $\L$ to $\Cc\times \{x\}$ is a line bundle of degree $d$ on $\Cc$, so by Riemann-Roch for gerbey curves (Corollary \ref{cor:gerbeyRRforcurves}), we have that 
$$
h^0(\Cc, \L|_{\Cc \times \{ x\}} ) = d.
$$
% with divisor $D$, so by Riemann-Roch for gerbes (which follows as a corollary of Grothendieck-Riemann-Roch for gerbes)
% \[
% \ell(D) +\ell(K-D)=\deg D-g+1=\deg D.
% \]
%\textcolor{red}{Can we use a RR for stacks directly? Now that we don't have the statement of 6.3, this is what I'm worried about: $x$ corresponds to a line bundle on $\Cc$. But this doesn't correspond to a line bundle on $C$, and so I don't know how to associate a divisor to it.}
%We have assumed that $d>0$.  Then $K-D$ is a line bundle of degree $-d$, and therefore it has no global sections.  This means that $\ell(D)=\deg D=d$, 
This gives us the number of sections of the pushforward $\pi_{*}\L.$ %\textcolor{red}{I just moved all of the stuff that we had here to its own corollary.}

We will use Grothendieck-Riemann-Roch to show that the degree of the pushforward is $d-1$. By Lemma \ref{lem:simplifiedgrrforcurves}, it suffices to compute the Chern character of $\L$ itself. Further, since we are only interested in the degree of $\pi_{*}\L,$ we will only calculate the top part of the Chern character, $c_1(\L)^2/2$. Consider the Weil divisor $D$ corresponding to $\L$, i.e. the codimension 1 substack of $\Cc \times \Cc'$ cut out by a section of $\L$. In what follows, by abuse of notation, we write the Weil divisors on $\Cc \times \Cc'$ corresponding to $\Cc$ and $\Cc'$ as $\Cc$ and $\Cc'$ as well. %\textcolor{red}{Is the class $[\Cc]$ independent of the choice of $x$ when we think of it as $\Cc \times \{ x\} \subset \Cc \times \Cc'$? $\Cc'$ is a genus 1 curve, so moving fibers is hard? Oh wait, the fibers are algebraically equivalent, and for $\CH^1$, those two are the same? Phew. Momentary heart attack.}

%(\textcolor{red}{This paragraph may not work any more}) The degree of a line bundle is preserved by base change.  Therefore we can compute the degree of $\pi_*\L$ on the coarse space, and that value will give its degree on the gerbe.  The rigidified $\L$ is a line bundle on a surface, so its chern character is $1+c_1+\frac{c_1^2}{2}$.  The top degree part is $\frac{c_1^2}{2}$, so it suffices to compute this value (which is an integer).  

%In order to compute the self-intersection $c_1^2$, we want to turn $c_1$ into a divisor on $\Cc \times \Cc'$, rather than a line bundle.  
Consider the morphism $\Cc \to \Cc'$ given by sending a point $p$ in $\Cc$ to $d\cdot p$ in $\Pic^d(\Cc)$.  Let $\Delta$ denote the graph of this morphism.  Then $c_1(D)$ is equal to $\Delta+(d-1)\Cc$, where $\Cc$ denotes the fiber class of $\Cc$ over $\Cc'$.  See Figure 1 for a diagram of this divisor, where the diagonal line represents $\Delta$ and the horizontal lines represent the $d-1$ copies of $\Cc$.  Each copy of $\Cc$ intersects $\Delta$ in 1 point, and the copies of $\Cc$ all intersect each other in zero points.  Therefore $c_1^2=(\Delta+(d-1)\Cc).(\Delta+(d-1)\Cc)=2d-2$.  This means that the degree is $d-1$, so $\pi_*(\L)$ has rank $d$ and degree $d-1$.  Therefore the dual of $\L$, which we denote $V$, induces an equivalence $\Phi^V:D(\Cc') \to D(\Cc)$, which inverts the derived equivalence $D(\Cc) \to D(\Cc')$.  $V$ has rank $d$ and degree $1-d$ over $\Cc'$, which implies that $\Cc$ is the moduli space of weight $1$ rank $d$ degree $1-d$ vector bundles on $\Cc'$ and $V$ is the universal bundle for this moduli problem.

\end{proof}

When $d<0$, note that the dual of the universal line bundle produces a derived equivalence in this case, that is, $\Phi^{\L^\vee}:D(\Cc) \to D(\Cc')$ is an equivalence.  Applying the same argument as above shows that $\Cc$ is the moduli space of weight $1$ rank $d$ degree $d+1$ vector bundles on $\Cc'$.

\begin{figure}[htb]
\begin{center}
\includegraphics[height=2in,width=2in,angle=0]{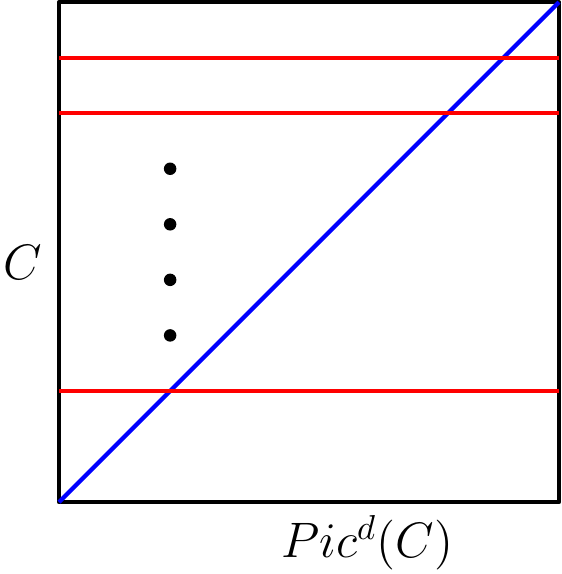}
\caption{Divisor of the universal degree d line bundle.}
\end{center}
\end{figure}

\paragraph{A geometric perspective on the kernel of the equivalence.} Since the rank and degree are coprime, the coarse moduli space of rank $d$ degree $1-d$ sheaves is geometrically isomorphic to the moduli space of degree $1-d$ line bundles, as the determinant map is geometrically an isomorphism by Theorem~\ref{thm:atiyah}. This means that $\Cc$ can alternatively be considered as being geometrically the moduli space of degree $1-d$ line bundles on $\Cc'$. There are two distinct derived equivalences from $\Cc'$ to $\Cc$ which arise naturally from viewing $\Cc'$ as a moduli space of vector bundles of specified rank and degree on $\Cc$.  The first has kernel $V$, which is the universal vector bundle of rank $d$ and degree $1-d$.  The second has kernel the universal degree $1-d$ line bundle, call it $\L$, obtained via the determinant map. Orlov's representability theorem in ~\cite[Thm 2.2]{orlov} states that any fully faithful exact functor between derived categories is an integral functor and that any such functor corresponds to a unique kernel.  Therefore the integral functors corresponding to these different kernels must give different derived equivalences, which implies that they should differ by pre- or post-composing with a derived autoequivalence of $\Cc$ or $\Cc'$.  The resulting autoequivalence here admits a natural description: let $\L$ be the universal line bundle on $\Cc \times \Cc'$, $V$ be the universal rank $d$ vector bundle on $\Cc \times \Cc'$, and $\det:J_{\Cc'}(d,1-d) \to J_{\Cc'}(1,d-1)$ the determinant map as in Theorem~\ref{thm:atiyah}.  Then $\Phi^\L=\det \circ \Phi^V$, i.e., we get the integral transform with kernel $\L$ by composing the integral transform with kernel $V$ with the determinant map, which is an isomorphism on $C$ and therefore induces a derived autoequivalence of $C$.  Since $\gm$ acts by weight $1$ on the universal rank $d$ degree $1-d$ vector bundle, it acts with weight $d$ on the determinant of $V$, which means that $\Cc'=\Pic_d^{1-d}(\Cc)$.  

Next, we show how to compose these equivalences in an explicit moduli-theoretic way.  That is, if $\Cc'=\Pic^d(C)$, and $\Cc''=\Pic^f(C')$, then what is the rank and degree of the convolution of the universal degree $d$ and $f$ line bundles after pushing forward to $C''$?

\begin{thm}%\label{thm:compositiondim1}
If $\Cc$ is a gerbey genus 1 curve, $\Cc':=\Pic^d_1(\Cc)$ and $\Cc'':=\Pic^f_1(\Cc')$, then $\Cc''=\Pic^{df-1}_2(\Cc)$.
\end{thm}

\begin{proof}

%Because degree is preserved by flat base change, it suffices to do the computation on the coarse spaces $C$, $C'$ and $C''$.  
The goal here is to compute the rank, degree, and weight of the convolution of the universal degree $d$ and degree $f$ line bundles.  This can be done by running an intersection theory computation on the gerbey threefold $\Cc \times \Cc' \times \Cc''$.  If $\L$ is the universal degree $d$ line bundle on $\Cc \times \Cc'$, and $\L'$ is the universal degree $f$ line bundle on $\Cc' \times \Cc''$, the kernel of the composition is $\pi_{02*}(\pi_{01}^*\L \otimes \pi_{12}^*\L')$.  

We use Grothendieck-Riemann-Roch as before.  The tangent bundles of everything in sight are trivial, so the Todd classes of the tangent bundles vanish.  Therefore we want to compute $c_1^3/6$, where $c_1$ is the divisor of $\pi_{01}^*\L \otimes \pi_{12}^*\L'$.

Since we understand the divisors on $\Cc\times \Cc'$ and on $\Cc' \times \Cc''$, we can understand them on the threefold.  $\L$ pulls back to $(\Delta_{01}+(d-1)\Cc)\times \Cc''$ and $\L'$ pulls back to $(\Delta_{12}+(f-1)\Cc') \times \Cc'$.  Therefore the divisor of $\pi_{01}^*\L \otimes \pi_{12}^*\L'$ is $(\Delta_{01}+(d-1)\Cc)\times \Cc''+(\Delta_{12}+(f-1)\Cc') \times \Cc'$.  Computing the third power of this divisor shows that the pushforward of this kernel to $\Cc''$ has rank $df$ and degree $df-1$.  The weight of the $\gm$-action is the sum of the weights of the $\gm$-actions on the universal degree $d$ and degree $f$ line bundles, which is $2$.

\end{proof}

As a sanity check, note that this means that if we take degree $d$ line bundles on $\Cc$, and then degree 1 line bundles on $\Pic^d(\Cc)$ (whose universal line bundle induces the identity map on $D(\Pic^d(\Cc))$, the convolution of the kernels gives a sheaf of rank $d$ and degree $d-1$ on $\Cc''$, which is what we expect from the first computation.

\begin{rem} 
We could also apply these techniques to study the derived categories of gerbey K3s, or gerbey moduli spaces of higher dimension.  In particular, if $X$ is a K3 and $\M$ is a 1-dimensional component of the moduli stack of semistable sheaves such that $\M$ does not parameterize any strictly semistable sheaves, then $D(X \times \bgm) \cong D(\M)$.  
\end{rem}

\end{document}